\documentclass[a4paper,10pt]{amsart}
 \usepackage{amsfonts,mathrsfs}
 \usepackage{amsrefs}
 \usepackage{amsthm}
 \usepackage{amssymb}
 \usepackage{enumerate}
 \usepackage{graphicx}
 \usepackage{tikz-cd}
 \usepackage{tikz-3dplot}
 \usepackage{cleveref}
 \usepackage{cancel}

 \allowdisplaybreaks

\theoremstyle{definition}
\newtheorem{Def}{Definition}[section]

\newtheorem{ex}{Example}
\newtheorem{remark}[Def]{Remark}
\theoremstyle{plain}

\newtheorem{theorem}[Def]{Theorem}
\newtheorem{proposition}[Def]{Proposition}

\newtheorem{lemma}[Def]{Lemma}
\usepackage{url}
\usepackage{kbordermatrix}

\usepackage{enumitem}
\numberwithin{equation}{section}

\allowdisplaybreaks

\makeatletter
\newcommand{\mylabel}[2]{#2\def\@currentlabel{#2}\label{#1}}
\makeatother

\usepackage{mathrsfs}
\usepackage{bbm}

\newcommand{\Int}{\operatorname{int}}

\newcommand{\Adj}{\mathcal{A}dj}
\newcommand{\dd}{{\rm d}}
\newcommand{\mj}{\mathbb{J}}
\newcommand{\me}{\mathbb{E}}

\newcommand{\x}{\underline{x}}
\newcommand{\uu}{\underline{u}}
\newcommand{\y}{\underline{y}}

\newcommand{\RR}{\mathbb R}
\newcommand{\NN}{\mathbb N}

\newcommand{\calB}{\mathcal B}

\newcommand{\calD}{\mathcal D}

\newcommand{\calL}{\mathcal L}

\newcommand{\calU}{\mathcal U}

\renewcommand{\kbldelim}{(}
\renewcommand{\kbrdelim}{)}
\renewcommand{\kbrowstyle}{\displaystyle}
\renewcommand{\kbcolstyle}{\displaystyle}

\DeclareMathOperator{\ext}{ext}

\DeclareMathOperator{\vol}{vol}

\newcommand{\cB}{{\mathcal B}}
\newcommand{\cC}{{\mathcal C}}
\newcommand{\cD}{{\mathcal D}}

\newcommand{\cF}{{\mathcal F}}

\newcommand{\cI}{{\mathcal I}}

\newcommand{\cQ}{{\mathcal Q}}

\newcommand{\ba}{{\mathbf a}}
\newcommand{\bb}{{\mathbf b}}
\newcommand{\bc}{{\mathbf c}}

\title[Extreme volumes for quasi-copulas]{Extreme values of the mass distribution associated with $d$-quasi-copulas via linear programming}
\author{Matej Bel\v sak}
\address{Matej Bel\v sak, 
Faculty of Computer and Information Science, University of Ljubljana.}
\email{mb4390@student.uni-lj.si}
\author{Matja\v z Omladi\v c${}^{2}$}
\address{Matja\v z Omladi\v c,
Faculty of Mathematics and Physics, University of Ljubljana  \& Institute of Mathematics, Physics and Mechanics, Ljubljana, Slovenia.} 
\email{matjaz@omladic.net}\thanks{${}^2$Supported by the ARIS (Slovenian Research and Innovation Agency)
research core funding No. P1-0448.}
\author{Martin Vuk}
\address{Martin Vuk, 
Faculty of Computer and Information Science, University of Ljubljana.}
\email{martin.vuk@fri.uni-lj.si}
\author{Alja\v z Zalar${}^{4}$}
\address{Alja\v z Zalar, 
Faculty of Computer and Information Science, University of Ljubljana  \& 
Faculty of Mathematics and Physics, University of Ljubljana  \&
Institute of Mathematics, Physics and Mechanics, Ljubljana, Slovenia.}
\email{aljaz.zalar@fri.uni-lj.si}
\thanks{${}^4$Supported by the ARIS (Slovenian Research and Innovation Agency)
research core funding No. P1-0288 and grants No.\ J1-50002, J1-60011.}

\newcommand{\comment}[1]{}
\newcommand{\mi}{\mathbb{I}}
\newcommand{\sign}{\mathrm{sign}}

\begin{document}


\date{\today}

\begin{abstract}
 The recent survey \cite{ARIASGARCIA20201} nicknamed ``Hitchhiker's Guide'' has raised the rating of quasi-copula problems in the dependence modeling community in spite of the lack of statistical interpretation of quasi-copulas. 
 This paper concentrates on the Open Problem 5 of this list concerning bounds on the volume of a $d$--variate quasi-copula. We disprove a recent conjecture \cite{UF23} on the lower bound of this volume. We also give evidence that the problem is much more difficult than suspected and give some hints towards its final solution.
\end{abstract}
\maketitle

\section{Introduction}
\label{sec:intro}


The role of copulas in dependence modeling is growing ever since Sklar \cite{Sklar1959229} in 1959 observed their omnipotence. We can think of copula as a multivariate distribution with uniform margins, but when we insert into it arbitrary univariate distributions as margins, we get an arbitrary multivariate distribution and any distribution of a random vector can be acquired in this way. More details about copulas can be found in \cite{Nel05,durante2015principles}. Now, quasi-copulas have become their silent companions: they have no statistical interpretation because they may have negative volumes of certain boxes. While the upper Fr\'{e}chet-Hoeffding bound $M$, the pointwise supremum of all $d$-variate copulas, is always a copula, the lower bound $W$, the pointwise infimum of all $d$-variate copulas, is in general a quasi-copula. In case $d=2$ it is also a copula, but in case $d>2$ it is not. Since 1993 when quasi-copulas were first introduced in \cite{Als93} they have come a long way until the ``hitchhiker’s guide'' of 2020 \cite{ARIASGARCIA20201} with many references and a list of open problems.

Copula theory has been gaining in popularity perhaps due to a wide range of applications, from natural sciences and engineering through economics and finance and recently ecology and sustainable development. Their ability to describe a variety of different relationships among random variables has become indispensable. The applications of quasi-copulas mostly follow those of copulas, with some additional ones in fuzzy set theory and the theory of aggregation functions (e.g., \cites{2006217,DEBAETS20061463,KOLESAROVA20121}), as well as fuzzy logic \cite{HM08}.
This motivated further research on algebraic properties of quasi-copulas (e.g., \cites{NQMRLUF02, RODRIGUEZLALLENA2009717, NELSEN2005583, FERNANDEZSANCHEZ20111365, FERNANDEZSANCHEZ2014109}), in particular also the question of extreme values of volumes of rectangles \cites{NQMRLUF02,BMUF07} which is the main topic of this work.
We start with the usual algebraic definition of quasi-copulas.
\\

Let $\calD\subseteq [0,1]^d$ be a non-empty set and $d\in \NN$, 
$d\geq 2$. We say that a function $Q:\calD\to [0,1]$ satisfies:
\begin{enumerate}
    \item[\mylabel{BC}{(BC)}] 
    \emph{Boundary condition:}
        If $\underline u:=(u_1, \ldots, u_d)\in \calD$ is of the form:
        \begin{align*} 
        &(a)\quad 
        \underline u=(u_1, \ldots, u_{i-1}, 0, u_{i+1}, \ldots, u_d)
        \text{ for some }i, 
        \text{then }Q(\underline u)=0.\\
        &(b)\quad 
        \underline u=(1, \ldots, 1, u_i, 1, \ldots, 1)
        \text{ for some }i, 
        \text{then }Q(\underline u)=u_i.
        \end{align*}
    \item[\mylabel{Mon}{(MC)}]  
    \emph{Monotonicity condition:} 
    $Q$ is nondecreasing in every variable, i.e., for each $i=1,\ldots,d$ and each pair of $d$-tuples
    \begin{align*}
    \begin{split}
    \uu&:=(u_1, \ldots,u_{i-1},u_i,u_{i+1},\ldots,u_d)\in \calD,\\
    \widetilde\uu&:=(u_1, \ldots,u_{i-1},\widetilde u_i,u_{i+1},\ldots,u_d)\in \calD,
    \end{split}
    \end{align*}
    such that $u_i\leq \widetilde u_i$, it follows that
    $Q(\uu)\leq Q(\widetilde\uu)$.
    \item[\mylabel{LC}{(LC)}] 
    \emph{Lipschitz condition}:
        Given $d$--tuples 
        $(u_1, \ldots, u_d)$ and $(v_1, \ldots, v_d)$ 
        in $\calD$ it holds that
        \begin{align*}
        |Q(u_1, \ldots, u_d) - Q(v_1, \ldots, v_d)| \leq \sum_{i = 1}^{d}|u_i - v_i|.
        \end{align*}
\end{enumerate}
If $\calD=[0,1]^d$ and $Q$ satisfies \ref{BC}, \ref{Mon}, \ref{LC}, then $Q$ is called a \emph{$d$-variate quasi-copula} (or \emph{$d$-quasi-copula}). We will omit the dimension $d$ when it is clear from the context and write quasi-copula for short.

Let $Q$ be a quasi-copula and $\cB=\prod_{i=1}^d [a_i,b_i]\subseteq [0,1]^d$ a $d$-box with $a_i<b_i$
for each $i$.
We will use multi-indices
of the form $\mi := (\mi_1, \mi_2 \ldots \mi_d) \in \{0, 1\}^d$ to index 
$2^d$ elements $\prod_{i=1}^d\{a_i,b_i\}$
of $\cB$.
Let $\|\mi\|_1:=\sum_{j=1}^d \mi_j$ be the $1$-norm of $\mi$.
We write 
\begin{equation*}
    x_\mi:=((x_\mi)_1,\ldots,(x_\mi)_d) 
\end{equation*}
to denote the vertex with coordinates
\begin{equation*}
    (x_\mi)_k = \begin{cases}
    a_k,\quad \text{if }\mi_k=0,\\
    b_k,\quad \text{if }\mi_k=1.
    \end{cases}
\end{equation*}
Let us denote the value of $Q$ in the point $x_\mi$
by 
\begin{equation*}
    q_\mi := Q(x_\mi).
\end{equation*}
(For an ilustration of our notation in dimensions 
$d=2$ and $d=3$, see Figure \ref{fig:notation} below.)
We write $\sign(\mi) := (-1)^{d-\|\mi\|_1}$. 
The \textbf{$Q$-volume} of $\cB$ is defined by:
\begin{equation*}
    V_Q(\cB) = \sum_{\mi\in \{0, 1\}^d} \sign(\mi) \cdot q_\mi.
\end{equation*}
\bigskip

\begin{figure}
\label{fig:notation}
    \begin{minipage}{0.5\textwidth}
        \centering
        \begin{tikzpicture}
    
            \draw[dotted, thin] (0, 0) rectangle (4, 4);
            
            \coordinate (A) at (1, 1.5);  
            \coordinate (B) at (3, 1.5);  
            \coordinate (C) at (3, 2.5);  
            \coordinate (D) at (1, 2.5);  
            
            \draw[thick] (A) rectangle (C);
            
            \node at (A) [below left] {$q_{0,0}$};
            \node at (B) [below right] {$q_{0,1}$};
            \node at (C) [above right] {$q_{1,1}$};
            \node at (D) [above left] {$q_{1,0}$};
            
            \node at (1, -0.3) {$a_1$};
            \draw (1, -0.1) -- (1, 0.1);
            \node at (3, -0.3) {$b_1$};
            \draw (3, -0.1) -- (3, 0.1);
            \node at (-0.3, 1.5) {$a_2$};
            \draw (-0.1, 1.5) -- (0.1, 1.5);
            \node at (-0.3, 2.5) {$b_2$};
            \draw (-0.1, 2.5) -- (0.1, 2.5);
            \node at (-0.3, 4) {$1$};
            \draw (-0.1, 2.5) -- (0.1, 2.5);
            \node at (4, -0.3) {$1$};
            \draw (3, -0.1) -- (3, 0.1);
            
            \draw[->, thin] (-0.5, 0) -- (4.5, 0);
            \draw[->, thin] (0, -0.5) -- (0, 4.5);
        \end{tikzpicture}
    \end{minipage}\hfill
    \begin{minipage}{0.5\textwidth}
        \centering
    
        \begin{tikzpicture}
        
        \pgfmathsetmacro{\one}{4}
        \draw[dotted, thin] (0, 0, 0) -- (\one, 0, 0) -- (\one, \one, 0) -- (0, \one, 0) -- cycle;
        \draw[dotted, thin] (0, 0, \one) -- (\one, 0, \one) -- (\one, \one, \one) -- (0, \one, \one) -- cycle;
        \draw[dotted, thin] (0, 0, 0) -- (0, 0, \one);
        \draw[dotted, thin] (\one, 0, 0) -. (\one, 0, \one);
        \draw[dotted, thin] (\one, \one, 0) -. (\one, \one, \one);
        \draw[dotted, thin] (0, \one, 0) -. (0, \one, \one);
        \pgfmathsetmacro{\aa}{1.5}
        \pgfmathsetmacro{\ab}{1.5}
        \pgfmathsetmacro{\ac}{1.5}
        \pgfmathsetmacro{\ba}{3.2}
        \pgfmathsetmacro{\bb}{3.2}
        \pgfmathsetmacro{\bc}{3.2}
        
        \coordinate (A) at (\aa, \ab, \ac); 
        \coordinate (B) at (\ba, \ab, \ac); 
        \coordinate (C) at (\ba, \bb, \ac); 
        \coordinate (D) at (\aa, \bb, \ac); 
        \coordinate (E) at (\aa, \ab, \bc); 
        \coordinate (F) at (\ba, \ab, \bc); 
        \coordinate (G) at (\ba, \bb, \bc); 
        \coordinate (H) at (\aa, \bb, \bc); 
        
        \draw[thick] (A) -- (B) -- (C) -- (D) -- cycle;
        \draw[thick] (E) -- (F) -- (G) -- (H) -- cycle;
        \draw[thick] (A) -- (E);
        \draw[thick] (B) -- (F);
        \draw[thick] (C) -- (G);
        \draw[thick] (D) -- (H);
        
        \node at (\aa + 0.5, \ab + 0.2, \ac) {$q_{0,0,0}$};
        \node at (\ba + 0.5, \ab + 0.1, \ac) {$q_{1,0,0}$};
        \node at (\aa + 0.5, \bb + 0.2, \ac) {$q_{0,1,0}$};
        \node at (\ba + 0.5, \bb + 0.1, \ac) {$q_{1,1,0}$};
        \node at (\aa + 0.7, \ab + 0.2, \bc)  {$q_{0,0,1}$};
        \node at (\ba + 0.5, \ab + 0.0, \bc){$q_{1,0,1}$};
        \node at (\aa + 0.7, \bb + 0.2, \bc) {$q_{0,1,1}$};
        \node at (\ba +0.5, \bb, \bc) {$q_{1,1,1}$};
        
        \node at (\aa, -0.3, 0) {$a_1$};
        \draw (\aa, -0.1, 0) -- (\aa, 0.1, 0);
        \node at (\ba, -0.3, 0) {$b_1$};
        \draw (\ba, -0.1, 0) -- (\ba, 0.1, 0);
        \node at (-0.3, \ab, 0) {$a_2$};
        \draw (-0.1, \ab, 0) -- (0.1, \ab, 0);
        \node at (-0.3, \bb, 0) {$b_2$};
        \draw (-0.1, \bb, 0) -- (0.1, \bb, 0);
        \node at (-0.3, 0, \ac) {$a_3$};
        \draw (-0.1, 0, \ac) -- (0.1, 0, \ac);
        \node at (-0.3, 0, \bc) {$b_3$};
        \draw (-0.1, 0, \bc) -- (0.1, 0, \bc);
        \node at (-0.3, 0, \one) {$1$};
        \node at (-0.3, \one, 0) {$1$};
        \node at (\one, -0.3, 0) {$1$};
        
        \draw[->, thin] (0,0,0) -- (4.5,0,0);
        \draw[->, thin] (0,0,0) -- (0,4.5,0);
        \draw[->, thin] (0,0,0) -- (0,0,4.5);
        
        \end{tikzpicture}
    \end{minipage}
    \caption{An illustration of the notation we use to denote the vertices of the box $\cB$ and the values of the $d$-quasi-copula $Q$ in those vertices. For $d=2$ the box $\cB = [a_1, b_1]\times [a_2, b_2]$ has four vertices $x_{0,0} = (a_1, a_2)$, $x_{1,0} = (b_1, a_2)$, $x_{0,1} = (a_1, b_2)$, $x_{1,1} = (b_1, b_2)$ that are indexed by four multi-indices $\mi^{(1)} = (0,0)$, $\mi^{(2)} = (1,0)$, $\mi^{(3)} =(0,1)$ and $\mi^{(4)} = (1, 1)$. The value of $Q$ at the vertex $x_{\mi^{(k)}}$ is denoted by $q_{\mi^{(k)}}$. For $d=2$ we have $q_{0,0} = Q(a_1, a_2)$, $q_{1,0} = Q(b_1, a_2)$, $q_{0,1} = Q(a_1, b_2)$, $q_{1,1} = Q(b_1, b_2)$, while for $d=3$ we have
    $q_{0,0,0} = Q(a_1, a_2, a_3)$, $q_{1,0, 0} = Q(b_1, a_2, a_3)$, $q_{0,1,0} = Q(a_1, b_2, a_3)$, $q_{0, 0, 1} = Q(a_1, a_2, b_3)$,
    $q_{1,1,0} = Q(b_1, b_2, a_3)$, $q_{1,0, 1} = Q(b_1, a_2, b_3)$, $q_{1,1,0} = Q(b_1, b_2, a_3)$, $q_{1, 1, 1} = Q(b_1, b_2, b_3)$.}
    \label{fig:enter-label}
\end{figure}

In this paper we fix a dimension $d\ge2$ and consider extreme values of the $Q$-volume $V_Q(\cB)$ over all quasi-copulas $Q$ and all $d$-boxes $\cB$. This is called Open Problem 5 on the list of the ``hitchhiker’s guide'' \cite{ARIASGARCIA20201}. It has recently been conjectured \cite{UF23} that the minimal value is $\displaystyle-\frac{(d-1)^2}{2d-1}$, attained for some $Q$ on
$\displaystyle\cB=\Big[\frac{d-1}{2d-1},\frac{2d-2}{2d-1}\Big]^d$. This conjecture is based on the results for $d=2$ \cite{Nelsen2002}, proved analytically, $d=3$ \cite{BMUF07} and $d=4$ \cite{UF23}, proved via linear programming using software tool \textit{Mathematica}.
We were able to extend the linear programming approach to solve \cite[Open Problem 5]{ARIASGARCIA20201} up to $d=9$ using \textit{Mathematica} \cite{ram2024} and programming language \textit{Julia} \cite{bezanson_julia_2017} up to $d=18$ for the minimal values and $d=17$ for the maximal values. (This is due to computer memory limits.) Our results disprove the conjecture of \cite{UF23}. Indeed, based on the results for $d\leq 18$ (resp.\ $d\leq 17$) the minimal and maximal values seem to follow the exponential function of the form $c2^{d}+d$ for suitable $c,d\in \RR$. For the maximal value the boxes $\cB$ are of the form $[a,1]^d$ for $2\leq d\leq 17$, while for minimal values this is true for $7\leq d\leq 18$.\\

The organization of the paper is the following. In Section \ref{preparatory} we prove that it suffices to relax the problem of extremes of $V_Q(\cB)$ over all $d$-quasi-copulas to the problem of extremes of functions $Q$ defined on a $2^d$-element grid of the vertices of $\cB$ or a $3^d$--element grid obtained by allowing values 1 for the coordinates of points, such that three types of conditions coming from \ref{BC}, \ref{Mon}, \ref{LC} are satisfied (see Theorems \ref{main} and \ref{main-v2}). The proof of these results is constructive by extending the function from the grid to the whole $d$-box $[0,1]^d$. These results enable us to formulate the problem in terms of two linear programs in Section 3 
(see \eqref{LP} and \eqref{LP-extension}). Numerical results obtained are presented in Tables \ref{tab:table1} (for the minimal values), \ref{tab:table2} (for the maximal values) and graphically on Figure \ref{fig:figure1}.


\section{Relaxation of the problem}
\label{preparatory}

The first main result of this section, Theorem \ref{main}, states that every function $Q$ defined on the vertices of a $d$-box $\prod_{i=1}^d [a_i,1]\subseteq [0,1]^d$ 
satisfying certain conditions coming from the definition of a quasi-copula,
extends to a quasi-copula. 
This enables us to formulate a relaxation of the problem of extremes of $V_Q(\cB)$ over all quasi-copulas $Q$ and all $d$-boxes $\cB$
as a linear program in Section \ref{sec:results}.

The second main result, Theorem \ref{main-v2}, is analogous to Theorem \ref{main},
where the $2^d$-element grid $\prod_{i=1}^d \{a_i,1\}$ is replaced by a $3^d$-element
grid $\prod_{i=1}^d \{a_i,b_i,1\}$. We need this result in Section \ref{sec:results} to solve those cases, where the relaxation of the volume problem does not give grids of the form
$\prod_{i=1}^d \{a_i,1\}$ as a result. This happens for minima of $V_Q(\cB)$ in dimensions $d\leq 6$.
\\

Fix $d\in \NN$, $d\geq 2$. For $s\in\NN$ we write 
\begin{align*} 
[s]
&:=\{0,1,\ldots,s\}
\quad\text{and}\quad
[-s;s]:=\{-s,-s+1,\ldots,-1,0,1,\ldots,s\}.
\end{align*}
For multi-indices 
$$\mi=(\mi_1,\ldots,\mi_d)\in [s]^d\quad \text{and} \quad\mj=(\mj_1,\ldots,\mj_d)\in [s]^d$$ 
let
    $$\mj-\mi=(\mj_1-\mi_1,\ldots,\mj_d-\mi_d)\in [-s;s]^d$$ 
stand for their usual coordinate-wise difference.
Let 
$\me^{(\ell)}$ stand for the multi-index with the only non-zero coordinate the $\ell$--th one, which is equal to 1.
For each $\ell=1,\ldots,d$ we define a relation $\sim_\ell$ on $[s]^d$ by
    $$
        \mi \sim_\ell \mj 
        \quad \Leftrightarrow \quad
        \mj-\mi=\me^{(\ell)}.    
    $$

For a tuple $\x=(x_1,\ldots,x_d)$ we define the functions
\begin{align*}
    G_d:\RR^d\to \RR,\quad 
    G_d(\x)
    &:=\sum_{i=1}^d x_i-d+1,\\
    H_d:\RR^d\to \RR,\quad
    H_d(\x)
    &:=\min\{x_1, x_2, \ldots, x_d\}.
\end{align*}
\smallskip

Let $\calD=\prod_{i=1}^d\{a_i,1\}$ be a $2^d$-element set with 
$0\leq a_i<1$ for each $i$.
For $\mi\in \{0,1\}^d$ let $x_\mi:=((x_\mi)_1,\ldots,(x_\mi)_d)\in\calD$, where  
\begin{equation}
\label{def:xI-v4}
    (x_\mi)_k = 
    \left\{
    \begin{array}{rl}
    a_k,\quad \text{if }\mi_k=0,\\
    1,\quad \text{if }\mi_k=1.
    \end{array}
    \right.
\end{equation}

The first main result of this section is the following.

\begin{theorem}
\label{main}
Fix $d\in \NN$, $d\geq 2$. 
Given 
\begin{itemize}
    \item nonnegative real numbers $a_{i}\in [0,1)$, $i=1,\ldots,d$,
    \item points $\displaystyle x_\mi\in \prod_{i=1}^d\{a_i,1\}$ defined by \eqref{def:xI-v4} and
    \item nonnegative real numbers 
        $q_\mi\in [0,1]$ for all $\mi\in \{0,1\}^d$,
\end{itemize} 
such that
\begin{align}
\label{th:main-cond-2}
\begin{split}
&
0\le q_{\mj} - q_{\mi} \le 1 - a_\ell
\quad 
    \text{for all }\ell=1,\ldots,d
    \;
    \text{and all }\mi\sim_\ell \mj,
    \\
&G_d(x_{\mi})
    \le q_\mi \le 
    H_d(x_{\mi})
    \quad \text{for all }\mi\in \{0,1\}^d,
\end{split}
\end{align}
there exists a quasi-copula $Q:[0,1]^d\to [0,1]$
satisfying
    $$
    Q(x_\mi)=q_\mi \quad \text{for all } \mi\in \{0,1\}^d.
    $$
\end{theorem}
\medskip


Let us demonstrate the statement of Theorem \ref{main}
in dimension $d=3$.

\begin{ex}
For $d=3$ we have 
\begin{align*}
    \{0,1\}^3
    &=    
    \big\{
        \underbrace{(0,0,0)}_{\mi^{(1)}},
        \underbrace{(1,0,0)}_{\mi^{(2)}},
        \underbrace{(0,1,0)}_{\mi^{(3)}},
        \underbrace{(0,0,1)}_{\mi^{(4)}},
        \underbrace{(1,1,0)}_{\mi^{(5)}},
        \underbrace{(1,0,1)}_{\mi^{(6)}},
        \underbrace{(0,1,1)}_{\mi^{(7)}},
        \underbrace{(1,1,1)}_{\mi^{(8)}},
    \big\}.
\end{align*}
Note that
\begin{align*}
    &\mi^{(1)}\sim_1 \mi^{(2)},
    &\mi^{(3)}&\sim_1 \mi^{(5)},
    &\mi^{(4)}&\sim_1 \mi^{(6)},
    &\mi^{(7)}&\sim_1 \mi^{(8)},\\
    &\mi^{(1)}\sim_2 \mi^{(3)},
    &\mi^{(2)}&\sim_2 \mi^{(5)},
    &\mi^{(4)}&\sim_2 \mi^{(7)},
    &\mi^{(6)}&\sim_2 \mi^{(8)},\\
    &\mi^{(1)}\sim_3 \mi^{(4)},
    &\mi^{(2)}&\sim_3 \mi^{(6)},
    &\mi^{(3)}&\sim_3 \mi^{(7)},
    &\mi^{(5)}&\sim_3 \mi^{(8)}.
\end{align*}
Given 
    \begin{align}
    \label{th:main-3-ex-cond-1}
    \begin{split}
        &a_1,a_2,a_3\in [0,1),\\
        &x_{\mi^{(1)}}=(a_1,a_2,a_3),\quad
        x_{\mi^{(2)}}=(1,a_2,a_3),\quad
        x_{\mi^{(3)}}=(a_1,1,a_3),\\
        &x_{\mi^{(4)}}=(a_1,a_2,1),\quad
        x_{\mi^{(5)}}=(1,1,a_3),\quad
        x_{\mi^{(6)}}=(1,a_2,1),\\
        &x_{\mi^{(7)}}=(a_1,1,1),\quad
        x_{\mi^{(8)}}=(1,1,1)
        \quad\text{and}\quad\\
        &q_{\mi^{(1)}},q_{\mi^{(2)}},q_{\mi^{(3)}},q_{\mi^{(4)}},
        q_{\mi^{(5)}},q_{\mi^{(6)}},q_{\mi^{(7)}},q_{\mi^{(8)}}
        \in [0,1]
    \end{split}
    \end{align}
conditions \eqref{th:main-cond-2} of Theorem \ref{main} are
\begin{align}
    \label{th:main-3-ex-cond-2}
    \begin{split}
    0&\leq q_{\mi^{(2)}}-q_{\mi^{(1)}}\leq 1-a_1,\quad
        0\leq q_{\mi^{(5)}}-q_{\mi^{(3)}}\leq 1-a_1,\\
    0&\leq q_{\mi^{(6)}}-q_{\mi^{(4)}}\leq 1-a_1,\quad
        0\leq q_{\mi^{(8)}}-q_{\mi^{(7)}}\leq 1-a_1,\\
    0&\leq q_{\mi^{(3)}}-q_{\mi^{(1)}}\leq 1-a_2,\quad
        0\leq q_{\mi^{(5)}}-q_{\mi^{(2)}}\leq 1-a_2,\\
    0&\leq q_{\mi^{(7)}}-q_{\mi^{(4)}}\leq 1-a_2,\quad
        0\leq q_{\mi^{(8)}}-q_{\mi^{(6)}}\leq 1-a_2,\\
    0&\leq q_{\mi^{(4)}}-q_{\mi^{(1)}}\leq 1-a_3,\quad
        0\leq q_{\mi^{(6)}}-q_{\mi^{(2)}}\leq 1-a_3,\\
    0&\leq q_{\mi^{(7)}}-q_{\mi^{(3)}}\leq 1-a_3,\quad
        0\leq q_{\mi^{(8)}}-q_{\mi^{(5)}}\leq 1-a_3,\\
    a_1+a_2+a_3-2
    &\leq q_{\mi^{(1)}}\leq \min\{a_1,a_2,a_3\},\\
    a_2+a_3-1
    &\leq q_{\mi^{(2)}}\leq \min\{a_2,a_3\},\\
    a_1+a_3-1
    &\leq q_{\mi^{(3)}}\leq \min\{a_1,a_3\},\\
    a_1+a_2-1
    &\leq q_{\mi^{(4)}}\leq \min\{a_1,a_2\},\\
    a_3
    &\leq q_{\mi^{(5)}}\leq a_3,\\
    a_2
    &\leq q_{\mi^{(6)}}\leq a_2,\\
    a_1
    &\leq q_{\mi^{(7)}}\leq a_1,\\
    1
    &\leq q_{\mi^{(8)}}\leq 1,\\
    \end{split}
\end{align}
The conclusion of Theorem \ref{main} is that for
    \eqref{th:main-3-ex-cond-1} 
satisfying
    \eqref{th:main-3-ex-cond-2},
there exists a quasi-copula $Q:[0,1]^2\to [0,1]$
such that 
\begin{align*}
    Q(a_1,a_2,a_3)&=q_{\mi^{(1)}},
        &Q(1,a_2,a_3)&=q_{\mi^{(2)}},
        &Q(a_1,1,a_3)&=q_{\mi^{(3)}},\\
    Q(a_1,a_2,1)&=q_{\mi^{(4)}},
        &Q(1,1,a_3)&=q_{\mi^{(5)}},
        &Q(1,a_2,1)&=q_{\mi^{(6)}},\\
    Q(a_1,1,1)&=q_{\mi^{(7)}},
        &Q(1,1,1)&=q_{\mi^{(8)}}.
\end{align*}

Note that
    $q_{\mi^{(5)}}=a_3$,
    $q_{\mi^{(6)}}=a_2$,
    $q_{\mi^{(7)}}=a_1$,
    $q_{\mi^{(8)}}=1,$ 
which must clearly hold, since $Q$ is a quasi-copula.
\hfill$\blacksquare$
\end{ex}
\bigskip


Let now $\calD=\prod_{i=1}^d\{a_i,b_i,1\}$ be a $3^d$-element set with 
$0\leq a_i<b_i<1$ for each $i$.
For $\mi\in \{0,1,2\}^d$ let $x_\mi:=((x_\mi)_1,\ldots,(x_\mi)_d)\in\calD$, where  
\begin{equation}
\label{def:xI-v3}
    (x_\mi)_k = 
    \left\{
    \begin{array}{rl}
    a_k,\quad \text{if }\mi_k=0,\\
    b_k,\quad \text{if }\mi_k=1,\\
    1,\quad \text{if }\mi_k=2.
    \end{array}
    \right.
\end{equation}

The second main result of this section is the following.

\begin{theorem}
\label{main-v2}
Fix $d\in \NN$,  $d\geq 2$. 
Given 
\begin{itemize}
\item nonnegative real numbers 
    $a_{i},b_i\in [0,1]$, $i=1,\ldots,d$, 
\item points $\displaystyle x_\mi\in \prod_{i=1}^d\{a_i,b_i,1\}$ defined by \eqref{def:xI-v3}
and
\item nonnegative real numbers $q_\mi\in [0,1]$ for all $\mi\in \{0,1,2\}^d$, 
\end{itemize}
such that
\begin{align}
\label{th:main-v2-cond-1}
\begin{split}
& 
    a_\ell<b_\ell\quad \text{for }\ell=1,\ldots,d,
    \\
&
0\le q_{\mj} - q_{\mi} \le b_\ell - a_\ell
\quad 
    \text{for all }\ell=1,\ldots,d
    \;
    \text{and all }\mi\sim_\ell\mj \text{ with }\mj_\ell=1,
    \\
&
0\le q_{\mj} - q_{\mi} \le 1 - b_\ell
\quad 
    \text{for all }\ell=1,\ldots,d
    \;
    \text{and all }\mi\sim_\ell\mj \text{ with }\mj_\ell=2,
    \\
&G_d(x_{\mi})
    \le q_\mi \le 
    H_d(x_{\mi})
    \quad \text{for all }\mi\in \{0,1,2\}^d,
\end{split}
\end{align}
there exists a quasi-copula $Q:[0,1]^d\to [0,1]$
such that 
    $$
    Q(x_\mi)=q_\mi \quad \text{for all } \mi\in \{0,1,2\}^d.
    $$
\end{theorem}

We will give constructive proofs of Theorems \ref{main} and \ref{main-v2}. First we establish some auxiliary results.\\

The next two lemmas are very basic and follow immediately from the triangle inequality. The first lemma states that the Lipschitz condition in the multidimensional case is equivalent to the Lipschitz condition in each dimension, while the second lemma reduces the Lipschitz condition in the univariate case to the piecewise Lipschitz condition. We mention that in literature the contents of these lemmas fall under the name \emph{local versus global Lipschitz condition} (e.g., \cite{Hei05}). However, usualy the underlying norm is the usual Euclidean, while we need results for the $\ell^1$ norm. Therefore we include the proofs of both lemmas for the sake of completeness. 

\begin{lemma}\label{lem:lipschitz-one-dim}
Let $\delta_1,\ldots,\delta_d$ be subsets of $[0,1]$ 
and
$\calD=\prod_{i=1}^d \delta_i\subseteq [0,1]^d$.
Let $F:\calD\to [0,1]$ be a function. 
The following statements are equivalent:
\begin{enumerate}
    \item\label{LC-pt1} 
        $F$ satisfies the Lipschitz condition on $\calD$.
    \item\label{LC-pt2}
        $F$ satisfies the Lipschitz condition in each variable,
        i.e., if 
        \begin{align*}
            \x&:=
            (a_1,\ldots,a_{i-1},x,a_{i+1},\ldots,a_d)\in \calD,\\
            \y&:=
            (a_1,\ldots,a_{i-1},y,a_{i+1},\ldots,a_d)\in \calD,
        \end{align*} 
        then
        \begin{equation*}
            |F(\x) - F(\y)| \le |x - y|. 
        \end{equation*}
\end{enumerate} 
\end{lemma}

\begin{proof}
    The nontrivial implication is 
    $\eqref{LC-pt2}\Rightarrow\eqref{LC-pt1}$.
    Let 
        $
        \x=(x_1,\ldots,x_d)\in \cD$
    and
        $
        \y=(y_1,\ldots,y_d)\in \cD.
        $
    We have that 
     \begin{align*}
            |F(\x) - F(\y)| 
            &= 
            \big|F(x_1, x_2, \ldots x_d) - F(y_1, x_2, \ldots, x_d) +\\
            &\hspace{1cm}F(y_1, x_2, \ldots, x_d) - F(y_1, y_2, x_3 \ldots, x_d) +\ldots+\\
            &\hspace{1cm}
            F(y_1, y_2, \ldots, y_{d-1},x_d)
            - F(y_1, y_2, \ldots, y_d)\big|\\
            &\leq
            \big|F(x_1, x_2, \ldots x_d) - F(y_1, x_2, \ldots, x_d)\big| +\\
            &\hspace{1cm}
            \big|F(y_1, x_2, \ldots, x_d) - F(y_1, y_2, x_3 \ldots, x_d)\big| +\ldots+\\
            &\hspace{1cm}
            \big|F(y_1, y_2, \ldots, y_{d-1},x_d)
            - F(y_1, y_2, \ldots, y_d)\big|\\
            &\le 
                |x_1 - y_1| + 
                |x_2 - y_2| + 
                \ldots +
                |x_d - y_d|,
    \end{align*}
    where we used the triangular inequality in the first inequality and the Lipschitz condition in each variable in the second inequality.
\end{proof}

\begin{lemma}\label{lem:lipschitz-piecewise}
    Let $x_1\le x_2\le \ldots \le x_n$ be points in $\RR$ and $f:[x_1,x_n]\to \RR$ a function. 
    The following statements are equivalent:
\begin{enumerate}
    \item\label{LC-piecewise-pt1} 
        $f$ satisfies the Lipschitz condition on $[x_1,x_n]$.
    \item\label{LC-piecewise-pt2}
        $f$ satisfies the Lipschitz condition on $[x_i,x_{i+1}]$ for $i=1,\ldots,n-1$.
\end{enumerate} 
\end{lemma}
\begin{proof}
    The nontrivial implication is 
    $\eqref{LC-piecewise-pt2}\Rightarrow\eqref{LC-piecewise-pt1}$.
    Let $x<y$ with $x,y\in [x_1,x_n]$.
    Let $k,\ell$ be integers such that
    $2\leq k<\ell\leq n$.
    Suppose that 
        $x\in [x_{k-1},x_k]$ 
    and 
        $y\in [x_{\ell},x_{\ell+1}]$.
    We have that 
    \begin{align*}
            |f(y) - f(x)| 
            &= 
            \big|f(x_k) - f(x) +
                f(x_{k+1}) - f(x_{k}) +\ldots+\\
            &\hspace{1cm}
            f(x_{\ell})-f(x_{\ell-1})
                +f(y)-f(x_{\ell})\big|\\
            &\leq
            \big|f(x_k) - f(x)\big| +
                \big|f(x_{k+1}) - f(x_{k})\big| +\ldots+\\
            &\hspace{1cm}
            \big|f(x_{\ell})-f(x_{\ell-1})\big|
                +\big|f(y)-f(x_{\ell})\big|\\
            &\le 
                (x_{k} - x) + 
                (x_{k+1} - x_{k}) + 
                \ldots +
                (x_{\ell} - x_{\ell-1})+
                (y-x_{\ell})\\
            &=y-x,
    \end{align*}
    where we used the triangular inequality in the first inequality and the piecewise Lipschitz condition
    from \eqref{LC-piecewise-pt2} in the second inequality.
\end{proof}

Recall that a multivariate function $F:\prod_{i=1}^d[a_i, b_i]\to \RR$,
where
$\prod_{i=1}^d[a_i, b_i]\subseteq [0,1]^d$, is \emph{multilinear} (resp.\ \emph{piecewise multilinear}), if it is linear (resp.\ piecewise linear)
separately in each variable, i.e., every restriction of $F$ to a single variable by fixing values of all other variables is linear (resp.\ piecewise linear).

The next proposition states that for a multilinear function on a $d$-box the Lipschitz condition is equivalent to the Lipschitz condition of its restriction to the vertices of the box.

\begin{proposition}\label{prop:ml-lipschitz}
    Let $F$
    be a multilinear function on a box $\cB=\prod_{i=1}^d[a_i, b_i]$, where $0\leq a_i<b_i\leq 1$ for each $i$. 
    The following statements are equivalent:
\begin{enumerate}
    \item\label{prop:ml-lipschitz-pt1} 
        $F$ satisfies the Lipschitz condition.
    \item\label{prop:ml-lipschitz-pt2} 
        $F$ satisfies the Lipschitz condition on the vertices, i.e., 
    \begin{equation*}
        |F(u_1, u_2,\ldots, u_d) - F(v_1, v_2, \ldots, v_d)| \le \sum_{i=1}^d |u_i - v_i|
    \end{equation*}
    for every choice of $u_i,v_i\in \{a_i,b_i\}$.
\end{enumerate} 
\end{proposition}
\begin{proof}
    The nontrivial implication is 
    $\eqref{prop:ml-lipschitz-pt2}\Rightarrow\eqref{prop:ml-lipschitz-pt1}$.
    We will prove by induction on $\ell$ that $F$ satisfies the Lipschitz condition
    on each $\ell$-dimensional face of $\prod_{i=1}^d[a_i,b_i]$ for $\ell=1,\ldots,d$.
    For $\ell=d$ we get the whole box $\cB$.
    Note that all $\ell$-dimensional faces are determined
    in the following way:
    \begin{enumerate}
        \item 
            Choose a set $\cI_\ell= \{i_1,i_2,\ldots,i_\ell\}$
            of $\ell$ coordinates 
            $1\leq i_1<i_2<\ldots<i_\ell\leq d$.
        \item 
            Let $\cI_\ell^c=\{j_1,\ldots,j_{d-\ell}\}$ be the remaining
            coordinates among the coordinates $1,\ldots,d$.
        \item 
            Fix 
            $c_{j}\in \{a_{j},b_{j}\}$
            for each $j\in \cI_\ell^c$.
        \item 
            The choices above determine a $\ell$-dimensional face
            $$
            \cF_{\cI_\ell,(c_{j_1},\ldots,c_{j_{d-\ell}})}:=\prod_{i=1}^d \delta_i,
            $$
            where 
            $$\delta_i=
            \left\{
            \begin{array}{rl}
                [a_i,b_i],&\text{if }i\in \cI_\ell,\\
                \{c_i\},&\text{if }i\in \cI_\ell^c.
            \end{array}
            \right.
            $$
    \end{enumerate}

    To prove the base of induction let $\ell=1$ and 
    $\cI_1$ be an arbitrary one-dimensional face of 
    $\prod_{i=1}^d [a_i,b_i]$, i.e., an edge of $\prod_{i=1}^d [a_i,b_i]$. By symmetry we may assume 
    without the loss of generality that $\cI_1=\{1\}$.
    So we need to prove that $F$ satisfies the Lipschitz condition on each edge
    $$
    \cF_{\{1\},\{c_2,\ldots,c_d\}}=
    \big\{
    (x,c_2, \ldots,c_d)
    \colon
    x\in [a_1, b_1]
    \big\},$$
    where 
    $c_i\in \{a_i, b_i\}$ for each $i$.
    Let $x,y\in [a_1,b_1]$ with $x\neq y$. 
    Then
    \begin{align*}
        &\big|
        F(x, c_2,\ldots c_d) - F(y,c_2, \ldots, c_{d})\big| = \\
        =&\left|
        \frac{
            F(b_1, c_2,\ldots,c_{d}) 
            - 
            F(a_1, c_2, \ldots,c_{d})}{b_1 - a_1}\right||x - y|\\
        \le& \frac{b_1-a_1}{b_1-a_1}|x-y|=|x-y|,   
    \end{align*}
    where in the first equality we used
    linearity of $F$ in the first variable and 
    in the inequality we used the Lipschitz condition 
    for a pair of vertices 
    $(b_1,c_2, \ldots,c_{d})$
    and 
    $(a_1,c_2, \ldots,c_{d})$.
    This proves the base of induction.

    Let us now assume $F$ is Lipschitz on each of the 
    $\ell$-dimensional faces of $\prod_{i=1}^d[a_i,b_i]$
    for $1\leq\ell\leq \ell_0-1$ where $1\leq \ell_0-1< d$,
    and prove $F$ is Lipschitz on all $\ell_0$-dimensional faces.
    By symmetry we may assume without the loss of generality that 
    $\cI_{\ell_0}=\{1,\ldots,\ell_0\}$.
    So we need to prove that $F$ satisfies the Lipschitz condition on each face of the form
    $$
    \cF\equiv
    \cF_{\cI_{\ell_0},\{c_{\ell_0+1},\ldots,c_d\}}=
    \big\{
    (x_1,\ldots,x_{\ell_0},c_{\ell_0+1}, \ldots,c_d)
    \colon
    x_i\in [a_i, b_i]\text{ for each }i
    \big\},$$
    where 
    $c_j\in \{a_j, b_j\}$ for each $j$.
    Using Lemma \ref{lem:lipschitz-one-dim}
    with $\cD=\cF$
    it suffices to prove the Lipschitz condition in each variable $i\in \cI_{\ell_0}$. By symmetry we may assume without the loss of generality that $i=1$.
    Let 
    \begin{align*}(x,z_2,\ldots,z_{\ell_0},c_{\ell_0+1},\ldots,c_d)
        &\in\cF,\\
    (y,z_2,\ldots,z_{\ell_0},c_{\ell_0+1},\ldots,c_d)
        &\in\cF.
    \end{align*}
    We have that each $z_{j}$ is of the form 
    $z_{j}=t_ja_{j}+(1-t_j)b_{j}$ for some $t_j\in [0,1]$.
    Then by linearity of $F$ in the $\ell_0$-th variable we have that
    \begin{align}
    \label{linearity}
    \begin{split}
    &F(z_1,z_2,\ldots,z_{\ell_0},c_{\ell_0+1},\ldots,c_d)
    =\\
    =&
    t_{\ell_0} F(z_1,z_2,\ldots,a_{\ell_0},c_{\ell_0+1},\ldots,c_d)
    +\\
    &\hspace{1cm}+
    (1-t_{\ell_0})   F(z_1,z_2,\ldots,b_{\ell_0},c_{\ell_0+1},\ldots,c_d).
    \end{split}
    \end{align}
    for each $z_1\in [a_1,b_1]$.
    Hence,
    \begin{align*}
    &\Big|F(x,z_2,\ldots,z_{\ell_0},c_{\ell_0+1},\ldots,c_d)-
    F(y,z_2,\ldots,z_{\ell_0},c_{\ell_0+1},\ldots,c_d)
    \Big|=\\
    =&
    \Big|t_{\ell_0} \big(F(x,z_2,\ldots,a_{\ell_0},c_{\ell_0+1},\ldots,c_d)
    -F(y,z_2,\ldots,a_{\ell_0},c_{\ell_0+1},\ldots,c_d)\big)+\\
    &+(1-t_{\ell_0}) \big(F(x,z_2,\ldots,b_{\ell_0},c_{\ell_0+1},\ldots,c_d)
    -F(y,z_2,\ldots,b_{\ell_0},c_{\ell_0+1},\ldots,c_d)\big)\Big|
    \leq
    \\
    \leq&
    t_{\ell_0}\Big|F(x,z_2,\ldots,a_{\ell_0},c_{\ell_0+1},\ldots,c_d)
    -F(y,z_2,\ldots,a_{\ell_0},c_{\ell_0+1},\ldots,c_d)\Big|\\
    &+(1-t_{\ell_0}) \Big|F(x,z_2,\ldots,b_{\ell_0},c_{\ell_0+1},\ldots,c_d)
    -F(y,z_2,\ldots,b_{\ell_0},c_{\ell_0+1},\ldots,c_d)\Big|\\
    \leq&t_{\ell_0}|x-y|+(1-t_{\ell_0})|x-y|=|x-y|,
    \end{align*}
    where we used \eqref{linearity} in the first equality,
    the triangle inequality in the first inequality
    and the Lipschitz condition for the $(\ell_0-1)$-dimensional faces
    $$\cF_{
    \{1,\ldots,\ell_{0}-1\},
    (a_{\ell_0},c_{\ell_0+1},\ldots,c_{d})
    }\quad\text{
    and}\quad
    \cF_{
    \{1,\ldots,\ell_{0}-1\},
    (b_{\ell_0},c_{\ell_0+1},\ldots,c_{d})
    },$$
    respectively.
    This proves the induction step and concludes the proof of the proposition.
\end{proof}

The next proposition states that for a multilinear function on a $d$-box the monotonicity condition is equivalent to the monotonicity condition of its restriction to the vertices of the box. The proof easily follows using multilinearity of $Q$, but we include it for the sake of completeness.

\begin{proposition}\label{prop:ml-monotonocity}
    Let $Q$
    be a multilinear function on a box $\cB=\prod_{i=1}^d[a_i, b_i]$, where $0\leq a_i<b_i\leq 1$ for each $i$. 
    The following statements are equivalent:
\begin{enumerate}
    \item\label{prop:ml-monotonocity-pt1} 
        $Q$ satisfies the monotonocity condition in each variable.
    \item\label{prop:ml-monotonocity-pt2} 
        $Q$ satisfies the monotonocity condition on the vertices, i.e., 
    \begin{equation*}
        Q(u_1, u_2,\ldots, u_d) \leq Q(v_1, v_2, \ldots, v_d),
    \end{equation*}
        where $u_i, v_i\in \{a_i, b_i\}$ and $u_i=v_i$ for all but at most one index $j$ where $u_{j}\leq v_j$.
\end{enumerate} 
\end{proposition}

\begin{proof}
    The nontrivial implication is 
    $\eqref{prop:ml-monotonocity-pt2}\Rightarrow\eqref{prop:ml-monotonocity-pt1}$.
    By symmetry it suffices to prove monotonicity in the first variable.
    Let $\x=(x,z_2,\ldots,z_d)$ and $\y=(y,z_2,\ldots,z_d)$ with 
    $a_1\leq x<y\leq b_1$ and
    $z_j\in [a_j,b_j]$ for each $j$.
    We have to prove that 
    \begin{equation}
        \label{mon-to-prove}
        Q(a_1,z_2,\ldots,z_d)\leq Q(b_1,z_2,\ldots,z_d).
    \end{equation}
    We have that 
        $x=t_xa_1+(1-t_x)b_1$, 
        $y=t_ya_1+(1-t_y)b_1$
    for some $0\leq t_y< t_x\leq 1$
    Each $z_j$ is of the form
        $z_j=t_ja_j+(1-t_j)b_j$ for some $t_j\in [0,1]$.
    Then
    \begin{align}
    \label{ml-monotonicity}
    \begin{split}
    &Q(z_1,z_2,\ldots,z_d)
    =\\
    =&
    \sum_{(r_1,\ldots,r_{d})\in [1]^{d}}
    t_1^{r_1}(1-t_1)^{1-r_1}
    \cdots 
    t_{d}^{r_{d}}
        (1-t_{d})^{1-r_{d}}\cdot\\
    & \hspace{1cm}
        \cdot Q(r_1a_1+(1-r_1)b_1,r_2a_2+(1-r_2)b_2
            \ldots,r_da_d+(1-r_d)b_d)
    \end{split}
    \end{align}
    Using \eqref{ml-monotonicity} for $z_1=a_1$ and $z_1=b_1$,
    it follows that \eqref{mon-to-prove} is equivalent to
    \begin{align}
    \label{ml-monotonicity-v2}
    \begin{split}
    &
    \sum_{(r_2,\ldots,r_{d})\in [1]^{d-1}}
    (t_y-t_x)
    t_{2}^{r_{2}}
        (1-t_{2})^{1-r_{2}}
    \cdots 
    t_{d}^{r_{d}}
        (1-t_{d})^{1-r_{d}}\cdot\\
    & \hspace{1cm}
        \cdot Q(b_1,r_2a_2+(1-r_2)b_2
            \ldots,r_da_d+(1-r_d)b_d)\\
    &\leq\sum_{(r_2,\ldots,r_{d})\in [1]^{d-1}}
    (t_y-t_x)
    t_{2}^{r_{2}}
        (1-t_{2})^{1-r_{2}}
    \cdots 
    t_{d}^{r_{d}}
        (1-t_{d})^{1-r_{d}}\cdot\\
    & \hspace{1cm}
        \cdot Q(a_1,r_2a_2+(1-r_2)b_2
            \ldots,r_da_d+(1-r_d)b_d)
    \end{split}
    \end{align}
    Since for each tuple $(r_2,\ldots,r_{d})\in [1]^{d-1}$
    we have 
    \begin{align*} 
        &Q(a_1,r_2a_2+(1-r_2)b_2,
            \ldots,r_da_d+(1-r_d)b_d)\\
        \leq&
        Q(b_1,r_2a_2+(1-r_2)b_2,
            \ldots,r_da_d+(1-r_d)b_d)
    \end{align*}
    by assumption \eqref{prop:ml-monotonocity-pt2} and since $t_y-t_x<0$,
    it follows that \eqref{ml-monotonicity-v2} holds.
    This proves the proposition.
\end{proof}

The next proposition is a first step toward establishing Theorems \ref{main} and \ref{main-v2} by defining $Q$ on $\calD$
and all $(d-1)$--dimensional faces 
\begin{equation}
\label{d-1-faces}
\calL_i=
\{
(x_1,\ldots,x_{i-1},0,x_{i+1},\ldots,x_d)
\colon x_i\in [0,1]
\}
\end{equation}
of $[0,1]^d$ containing $(0,\ldots,0)$. 

\begin{proposition}
\label{prop:extension}
Let one of the following assumptions hold:
\begin{enumerate}
    \item 
    \label{ass-1}
    $s=1$, 
    real numbers $a_{i}, $ $i=1,\ldots,d,$ 
    points $x_\mi\in \prod_{i=1}^d\{a_i,1\}=:\calD$ defined by \eqref{def:xI-v4}
    and real numbers 
    $q_\mi$ for all $\mi\in \{0,1\}^d$,
    satisfy conditions \eqref{th:main-cond-2} of Theorem \ref{main}.
    \item 
    \label{ass-2}
    $s=2$, 
    real numbers $a_{i}, b_i$, $i=1,\ldots,d,$
    points $x_\mi\in \prod_{i=1}^d\{a_i,b_i,1\}=:\calD$ defined by \eqref{def:xI-v3}
    and real numbers 
    $q_\mi$ for all $\mi\in \{0,1,2\}^d$,
    satisfy conditions \eqref{th:main-v2-cond-1} of Theorem \ref{main-v2}.
\end{enumerate}
Let
\begin{equation} 
\label{extended-domain}
\cD^{(\ext)}
    =
    \cD 
    \cup 
    \Big(\cup_{i=1}^d \calL_i\Big) 
\end{equation}
where
$\calL_i$ are as in \eqref{d-1-faces}.
Let 
$$
    Q:\calD^{(\ext)}\to\RR
$$ 
be defined by
\begin{equation}
    \label{def-Q}
    Q(x_1,\ldots,x_d)=\left\{
    \begin{array}{rl}
    q_\mi,& \text{if }(x_1,\ldots,x_d)=x_\mi \text{ for some }
        \mi\in [s]^d,\\
    0,& \text{if }(x_1,\ldots,x_d)\in \calL_i \text{ for some }i\in \{1,\ldots,d\},
    \end{array}
    \right.
\end{equation}
Then $Q$ is well-defined and satisfies 
\ref{BC}, \ref{Mon} and \ref{LC}.
\end{proposition}

\begin{figure}[h!]
        \centering
    
        \begin{tikzpicture}
        
        \pgfmathsetmacro{\one}{4}
        \pgfmathsetmacro{\aa}{2.5}
        \pgfmathsetmacro{\ab}{2.5}
        \pgfmathsetmacro{\ac}{2.5}
        \pgfmathsetmacro{\ba}{4}
        \pgfmathsetmacro{\bb}{4}
        \pgfmathsetmacro{\bc}{4}

        \usetikzlibrary{patterns}
        \draw[thick] (0, 0, 0) -- (\one, 0, 0) -- (\one, \one, 0) -- (0,\one, 0);
        \draw[thick] (0, 0, 0) -- (0, \one, 0) -- (0, \one, \one) -- (0, 0,\one);
        \draw[thick] (0, 0, 0) -- (0, 0, \one) -- (\one, 0, \one) -- (\one, 0,  0);
        \fill[gray,opacity=0.3] (0, 0, 0) -- (\one, 0, 0) -- (\one, \one, 0) -- (0,\one, 0);
        \fill[gray,opacity=0.3] (0, 0, 0) -- (0, \one, 0) -- (0, \one, \one) -- (0, 0,\one);
        \fill[gray,opacity=0.3] (0, 0, 0) -- (0, 0, \one) -- (\one, 0, \one) -- (\one, 0,  0);
        \node at (0.5, 3.5, 0) {$\calL_3$};
        \node at (2.5, 0.2, 2.5) {$\calL_2$};
        \node at (0.2, 2.5, 2.5) {$\calL_1$};
        
        \coordinate (A) at (\aa, \ab, \ac); 
        \coordinate (B) at (\ba, \ab, \ac); 
        \coordinate (C) at (\ba, \bb, \ac); 
        \coordinate (D) at (\aa, \bb, \ac); 
        \coordinate (E) at (\aa, \ab, \bc); 
        \coordinate (F) at (\ba, \ab, \bc); 
        \coordinate (G) at (\ba, \bb, \bc); 
        \coordinate (H) at (\aa, \bb, \bc); 
        
        \draw[thick] (A) -- (B) -- (C) -- (D) -- cycle;
        \draw[thick] (E) -- (F) -- (G) -- (H) -- cycle;
        \draw[thick] (A) -- (E);
        \draw[thick] (B) -- (F);
        \draw[thick] (C) -- (G);
        \draw[thick] (D) -- (H);
        
        \node at (\aa + 0.5, \ab + 0.2, \ac) {$q_{0,0,0}$};
        \node at (\ba + 0.5, \ab + 0.1, \ac) {$q_{1,0,0}$};
        \node at (\aa + 0.5, \bb + 0.2, \ac) {$q_{0,1,0}$};
        \node at (\ba + 0.5, \bb + 0.1, \ac) {$q_{1,1,0}$};
        \node at (\aa + 0.7, \ab + 0.2, \bc)  {$q_{0,0,1}$};
        \node at (\ba + 0.5, \ab + 0.0, \bc){$q_{1,0,1}$};
        \node at (\aa + 0.7, \bb + 0.2, \bc) {$q_{0,1,1}$};
        \node at (\ba +0.5, \bb, \bc) {$q_{1,1,1}$};
        
        \node at (\aa, -0.3, 0) {$a_1$};
        \draw (\aa, -0.1, 0) -- (\aa, 0.1, 0);
        \node at (\ba+0.5, -0.3, 0) {$b_1=1$};
        \draw (\ba, -0.1, 0) -- (\ba, 0.1, 0);
        \node at (-0.3, \ab, 0) {$a_2$};
        \draw (-0.1, \ab, 0) -- (0.1, \ab, 0);
        \node at (-0.7, \bb, 0) {$b_2=1$};
        \draw (-0.1, \bb, 0) -- (0.1, \bb, 0);
        \node at (-0.3, 0, \ac) {$a_3$};
        \draw (-0.1, 0, \ac) -- (0.1, 0, \ac);
        \node at (-0.7, 0, \bc) {$b_3=1$};
        \draw (-0.1, 0, \bc) -- (0.1, 0, \bc);
        
        \draw[->, thin] (0,0,0) -- (4.5,0,0);
        \draw[->, thin] (0,0,0) -- (0,4.5,0);
        \draw[->, thin] (0,0,0) -- (0,0,4.5);
        
        \end{tikzpicture}
    \caption{To construct a quasi-copula $Q$ on $[0,1]^3$ given the values $q_\mi$ at $x_\mi\in \calD:=\prod_{i=1}^3\{a_i,1\}$, we first define it to be 0 on all $2$-dimensional faces $\calL_1$, $\calL_2,$ $\calL_3$ and prove that the extension indeed
    meets the requirements of a quasi-copula. In the case $\calD:=\prod_{i=1}^3 \{a_i,b_i,1\}$
    the same applies, only that the initial grid consists of $3^3=27$ points, having values 
        $q_{\mi_1,\mi_2,\mi_3}$, $\mi_j\in \{0,1,2\}$.
    }
    \label{fig:facesL}
\end{figure}

\begin{proof}
First we prove well-definedness. 
We have to show that $x_\mi\in \calL_i$ implies that $q_\mi=0$. 
If $x_\mi\in \calL_i$, then $x_i=0$ and by \eqref{th:main-cond-2} we have that $0\leq q_\mi\leq H_d(x_\mi)=0$. Hence, $q_\mi=0$.

$Q$ satisfies \ref{BC} by definition of $Q$ on $\calL_i$ for each $i$.

By Lemma \ref{lem:lipschitz-one-dim}, it suffices to prove the Lipschitz condition separately for each variable.
In what follows we will prove monotonicity and the Lipschitz condition for each variable simultaneously. 
By symmetry it suffices to prove them for the first variable. Let us take $\x:=(x,a_2,\ldots,a_d)$ and $\y=(y,a_2,\ldots,a_d)$ with $0\leq x<y$ such that $\x,\y\in \calD^{(\ext)}$. We separate 3 cases:\\

\noindent \textbf{Case 1:} $\x,\y\in \cD$. 
Then $0\leq Q(\y)-Q(\x)\leq y-x$ by \eqref{th:main-cond-2} or 
\eqref{th:main-v2-cond-1}.\\

\noindent \textbf{Case 2:} $\x\in \cD$ and $\y\in\calL_i$
for some $i\in \{1,\ldots,d\}$. Then $i\neq 1$ (since $y>0$). But then $a_i=0$
and hence $\x\in \calL_i$ as well. It follows that
$Q(\x)=Q(\y)=0$.\\

\noindent \textbf{Case 3:} $\x\in \calL_i$ for some $i\in \{1,\ldots,d\}$. 

First note that $0=Q(\x)$. We separate two subcases according to the value of $i$.

Assume that $i=1$. If $\y\in \cD$, then $Q(\y)\leq y$ by \eqref{th:main-cond-2} or \eqref{th:main-v2-cond-1}. If $\y\in \calL_j$ for some $j$, then $Q(\y)=0$. In all cases \ref{Mon} and \ref{LC} are satisfied.

If $i>1$, then $\y\in\calL_i$ as well and $Q(\y)=0$.
\end{proof}

Now we are ready to prove Theorems \ref{main} and \ref{main-v2}.

\begin{proof}[Proof of Theorem \ref{main}]
Let $\calD$, $\calD^{(\ext)}$ and $Q$
be as in Proposition \ref{prop:extension} under the assumption \eqref{ass-1}.
We subdivide the box $[0,1]^d$
into $2^d$ smaller $d$-boxes 
\begin{equation}
\label{subboxes}
    \cB_{\mi}=\prod_{j=1}^d \delta_j(\mi)
\end{equation}
for $\mi=(\mi_1,\ldots,\mi_d)\in \{0,1\}^d$, where
    $$
    \delta_j(\mi)
    =
    \left\{
    \begin{array}{rl}
    [0,a_j],&   \text{if }\mi_j=0,\\
    \left[a_{j},1\right],& \text{if }\mi_j=1.
    \end{array}
    \right.
    $$
In particular,
\begin{align*}
    \calB_{(0,0\ldots 0)} &= \prod_{k=1}^d[0, a_k],
    \quad
    \calB_{(1, 0\ldots 0)}=[a_1, 1]\times \prod_{k=2}^d [0, a_k],\ldots,\quad
    \calB_{(1,1\ldots 1)}=\prod_{k=1}^d[a_k, 1].
\end{align*}

Note that the $Q$-volume $V_Q(\calB_\mi)$ of each box $\calB_\mi$ is determined by the value of $Q$
on the vertices of $\calB_\mi$.
For each $\mi\in \{0,1\}^d$ we define a constant function
    $$
    \rho_\mi:\cB_\mi\to \RR,\quad
    \rho_\mi:=\frac{V_Q(\cB_\mi)}{\prod_{j=1}^d \delta_j(\mi)}.
    $$
Let us define a piecewise constant function 
\begin{equation*}
\rho:[0,1]^d\to \RR,\quad
\rho(\x):=
\left\{
\begin{array}{rl}
\rho_\mi,&  \text{if }\x\in \Int(\cB_\mi) \text{ for some }\mi\in \{0,1\}^d,\\
0,& \text{otherwise},
\end{array}
\right.
\end{equation*}
where $\Int(A)$ stands for the topological interior of the set $A$ in the usual Euclidean topology.
We will prove that a function $Q:[0,1]^d\to \RR$,
defined by 
\begin{equation}
    Q(x_1, x_2, \ldots x_d) = \int_0^{x_1}\int_0^{x_2}\ldots \int_0^{x_d} 
    \rho(x_1, x_2, \ldots x_d) \dd x_1 \dd x_2\cdots \dd x_d,
    \label{eq:integral}  
\end{equation}
is a quasi-copula satisfying the statement of Theorem \ref{main}.
Note that in case $\rho$ is a density of a random vector, $Q$ is by definition its cumulative distribution function.

First observe that $Q$ is a piecewise multilinear function, 
which 
by construction 
extends $Q$ 
and hence $Q$ satisfies \ref{BC}.

Next we prove the Lipschitz condition.
Let $\cB_{\mi}$ be as in \eqref{subboxes}.
Since $Q$ coincides with $Q$ on $\calD^{(\ext)}$, 
the values of $Q$ at the vertices of any box $\cB_\mi$ 
coincide with the values of $Q$.
By Proposition \ref{prop:extension}, $Q$ satisfies the Lipschitz condition. Thus, Proposition \ref{prop:ml-lipschitz} implies that $Q$ satisfies the Lipschitz condition on each $B_\mi$. 
It remains to prove that $Q$ satisfies the Lipschitz condition on the whole box $[0, 1]^d$. 
By Lemma \ref{lem:lipschitz-one-dim}, it suffices to prove the Lipschitz condition for each variable.
By symmetry we may prove it only for the first variable.
Let us fix $x_2,\ldots,x_d\in [0,1]$ and let 
$f(x) := Q(x, x_2,\ldots, x_d)$, $x\in [0,1]$, be a function. By construction of $Q$, $f$ satisfies the Lipschitz condition on the intervals $[0, a_1]$, $[a_1, 1]$, since the Lipschitz condition is satisfied on the boxes $\calB_\mi$. By Lemma \ref{lem:lipschitz-piecewise}, it follows that $f$ satisfies the Lipschitz condition on the whole interval $[0, 1]$. This concludes the proof of the Lipschitz condition of $Q$.

It remains to prove the monotonicity of $Q$. Note that it is sufficient to prove that $Q$ is monotone on each box $\calB_\mi$.
Since on each of these boxes the function $Q$ is multilinear and for the vertices of the box the condition of monotonicity holds, 
monotonicity on the whole boxes follows by 
Proposition \ref{prop:ml-monotonocity}.
\end{proof}

\begin{proof}[Proof of Theorem \ref{main-v2}]
The proof is analogous to the proof of Theorem \ref{main}
only that the definitions of $\calD$, $\calD^{(\ext)}$ and $Q$
are replaced by the ones under the assumption \eqref{ass-2}
of Proposition \ref{prop:extension} and the box $[0,1]^d$
is subdivided 
into $3^d$ smaller $d$-boxes 
$\cB_{\mi}=\prod_{j=1}^d \delta_j(\mi)$ 
for $\mi=(\mi_1,\ldots,\mi_d)\in \{0,1,2\}^d$, where
    $$
    \delta_j(\mi)
    =
    \left\{
    \begin{array}{rl}
    [0,a_j],&   \text{if }\mi_j=0,\\
    \left[a_{j},b_j\right],& \text{if }\mi_j=1,\\
    \left[b_{j},1\right],& \text{if }\mi_j=2.
    \end{array}
    \right.
    $$
The construction of $Q$ and the arguments are then the same as in the proof of Theorem \ref{main}.
\end{proof}


\section{Solution to the problem in dimensions up to 17}
\label{sec:results}

Assume the notation as in Sections \ref{sec:intro} and \ref{preparatory}.
In this section we present solutions to \cite[Open Problem 5]{ARIASGARCIA20201} up to $d=18$ for the minimal 
    (see Table \ref{tab:table1} in Subsection \ref{subsec:minimal}) 
and up to $d=17$ for the maximal volume question
    (see Table \ref{tab:table2} in Subsection \ref{subsec:maximal}). 
The main technique to obtain our results is a  computer software approach
by solving the corresponding linear programs (\eqref{LP}, \eqref{LP-extension} for the minimal and \eqref{LP-max} for the maximal volume),
which are based on the main results of Section \ref{preparatory},
i.e., Theorems \ref{main} and \ref{main-v2}.
We conclude the section with a graphical presentation of solutions and concluding remarks (see Subsection \ref{subsec:concluding}).

\subsection{Minimal volume problem}
\label{subsec:minimal}
Let us define the following linear program:

\begin{align}
\label{LP}
\begin{split}
\min_{
\substack{
    a_1,\ldots,a_d,\\
    b_1,\ldots,b_d,\\
    q_{\mi} \text{ for }\mi\in \{0,1\}^d
}}
&\hspace{0.2cm} \sum_{\mi\in [1]^d} \sign(\mi)\cdot q_\mi,\\
\text{subject to }
&\hspace{0.2cm} 
    0\leq a_i< b_i\leq 1\quad i=1,\ldots,d,\\
&\hspace{0.2cm} 
0\le q_{\mj} - q_{\mi} \le b_\ell - a_\ell
\quad 
    \text{for all }\ell=1,\ldots,d
    \text{ and all }\mi\sim_\ell\mj,\\
&\hspace{0.2cm} 
    \max\{0,G_d(x_{\mi})\} 
    \le q_\mi \le 
    H_d(x_{\mi})
    \quad \text{for all }\mi\in \{0,1\}^d.
\end{split}
\end{align}

\begin{proposition}
    \label{relaxation}
    Let $\mi^{(1)},\ldots,\mi^{(2^d)}$ be some order of all multi-indices $\mi\in\{0,1\}^d$.
    If there exists an optimal solution 
    \begin{equation}
    \label{optimal-soln}
    (a_1^\ast,\ldots,a_d^\ast,
    b_1^\ast,\ldots,b_d^\ast,
    q_{\mi^{(1)}},\ldots,q_{\mi^{(2^d)}}
    )
    \end{equation}
    to \eqref{LP}, which
    satisfies 
    \begin{equation}
        \label{bi-equal-to-1}
            b_1^\ast=\ldots=b_d^\ast=1,
    \end{equation}
    then the optimal value of \eqref{LP} is the most negative volume over all boxes over all $d$-quasi-copulas.
\end{proposition}

\begin{proof}
    Use Theorem \ref{main} for the optimal solution 
    \eqref{optimal-soln} to conclude that this solution indeed extends 
    to some quasi-copula.
\end{proof}

By Proposition \ref{relaxation}, the linear program
\eqref{LP}
relaxes the problem of determining a $d$-quasi-copula $Q$,
such that the volume $V_Q(\cB)$ is minimal among all
$d$-quasi-copulas and all boxes $\cB\subseteq [0,1]^d$.
The relaxation refers to the fact, that if \eqref{bi-equal-to-1} does not hold, then the optimal solution to
\eqref{LP} might not extend to some quasi-copula and hence the optimal value of \eqref{LP}
only represents the lower bound for the minimal volume problem.\\

We wrote a function, which generates the linear program \eqref{LP} given dimension $d$, using software tools \textit{Mathematica} \cite{ram2024} and the modeling language \textit{JuMP} \cite{Jump2023} implemented in \textit{Julia} \cite{bezanson_julia_2017}. Then we used the simplex method in \textit{Mathematica} and a combination of the simplex and the interior point method using \textit{Julia} bindings \cite{Highs3} to \textit{HiGHS} solver \cites{Highs,Highs2} to solve it. The advantage of \textit{Mathematica} is that it does exact computations with rational numbers, while \textit{HiGHS} in \textit{Julia} uses floating point arithmetic. However, \textit{Mathematica} could compute the results up to dimension $d=9$ before running out of memory, while \textit{Julia} computed the results up to $d=18$ in reasonable time 
($55\text{s}$ for $d=16$, $241\text{s}$ for $d=17$ and $8712\text{s}$ for $d=18$).
In Table \ref{tab:table1} below we list an optimal solutions to \eqref{LP}.
For dimension $d>9$ the results are rounded to a rational number with absolute error lower than $10^{-6}$, while for $d\leq 9$ the results are exact from \textit{Mathematica}. The program has been run on a personal computer with \emph{AMD Ryzen 7 6800HS} processor and $16\mathrm{GB}$ of RAM.
The code and the results are publicly available in the Gitlab repository \cite{Quasi-copula-Gitlab}.

\begin{table}[h!]
  \begin{center}
      \caption{
      Minimal values of $V_Q(B)$ over all $d$--variate quasi--copulas $Q$ and all $d$--boxes $B\subseteq [0,1]^d$. 
      It turns out that the minimal box $B_{\min}$ is of the form $[a,b]^d$ and that the values $q_{\mi}$, $\mi=(\mi_1,\ldots,\mi_d)\in \{0,1\}^d$,
      depend only on 
      $\|\mi\|_1=\sum_{j=1}^d \mi_j$.
      We list $q_{\|\mi\|_1}:=q_\mi$.
      }
    \label{tab:table1}
    {\renewcommand{\arraystretch}{2}
    \begin{tabular}{|c|c|c|c|c|}
    \hline
    $d$ & $a$ & $b$ & 
    $\vec{q}=(q_{\|\mi\|_1})_{\|\mi\|_1=0}^d$ & $V_Q([a,b]^d)$ 
    \\
    \hline
    2 & $\frac 13$ & $\frac 23$ & 
    $\left(0,\frac{1}{3},\frac{1}{3}\right)$
    &
    $-\frac 13$  \\
    \hline
    3 & $\frac 25$ & $\frac 45$ &
    $\left(0,0,\frac{2}{5},\frac 25\right)$
    & $-\frac 45$  \\
    \hline
    4 & $\frac 37$ & $\frac 67$ & 
    $\left(0,0,0,\frac{3}{7},\frac 37\right)$
    &$-1\frac{2}{7}$\\
    \hline
    5 & $\frac{8}{13}$ & $\frac{12}{13}$ & 
    $\left(0,0,\frac{4}{13},\frac {4}{13},\frac{8}{13},\frac{8}{13}\right)$
    &
    $-2\frac{6}{13}$  \\
    \hline
    6 & $\frac 58$ & $\frac{15}{16}$ 
    &
    $\left(0,0,0,\frac{5}{16},
    \frac{5}{16},\frac{5}{8},
    \frac{5}{8}\right)$
    & $-4\frac{11}{16}$ \\
    \hline
    7 & $\frac 12$ & $1$ & 
    $\left(0,0,0,0,\frac{1}{2},
    \frac 12,\frac{1}{2},1\right)$
    &$-9\frac{1}{2}$  \\
    \hline
    8 & $\frac 23$ & $1$ & 
    $\left(0,0,0,\frac 13,\frac 13,\frac{2}{3},\frac 23,\frac{2}{3},1\right)$
    &
    $-18\frac 13$ \\
    \hline
    9 & $\frac 23$ & $1$ & 
    $\left(0,0,0,0,\frac 13,\frac 13,\frac{2}{3},\frac 23,\frac{2}{3},1\right)$
    & $-37$ \\
    \hline
    10 & $\frac 23$ & 1 & 
$\left(0,0,0,0,0,\frac 13,\frac 13,\frac{2}{3},\frac 23,\frac{2}{3},1\right)$
    &
    $ -69\frac{2}{3}$ \\\hline
    11 & $\frac 12$ & 1 & 
    $\left(0,0,0,0,0,0,\frac 12,\frac 12,\frac 12,\frac 12,\frac 12,1\right)$
    &
    $-125\frac{1}{2}$ \\\hline
    12 & $\frac 23$ & 1 & 
    $\left(0,0,0,0,0,\frac 13,
    \frac 13,\frac 23,\frac 23,
    \frac 23,\frac 23,\frac 23,1\right)$
    &
    $-263\frac{2}{3}$ \\\hline
    13 & $\frac 23$ & 1  & 
    $\left(
    0,0,0,0,0,0,\frac 13,
    \frac 13,\frac 23,
    \frac 23,\frac 23,\frac 23,
    \frac 23,1
    \right)$
    &
    $-527\frac{2}{3}$ \\\hline
    14 & $\frac 23$ & 1 
    &
    $\left(
    0,0,0,0,0,0,0,\frac 13,
    \frac 13, \frac 23,
    \frac 23,
    \frac 23,
    \frac 23,
    \frac 23,
    1
    \right)$
    & $-1000\frac{2}{3}$ \\\hline
    15 & $\frac 34$ & 1 & 
    $\left(
    0,0,0,0,0,0,
    \frac 14,
    \frac 14,
    \frac 12,
    \frac 12,
    \frac 34,
    \frac 34,
    \frac 34,
    \frac 34,
    \frac 34,
    1
    \right)$
    &$-1858\frac{3}{4}$ \\\hline
    16 & $\frac 23$ & 1 & 
    $\left(
    0,0,0,0,0,0,0,
    \frac 13,
    \frac 13,
    \frac 23,
    \frac 23,
    \frac 23,
    \frac 23,
    \frac 23,
    \frac 23,
    \frac 23,
    1
    \right)$
    & $-3813$ \\\hline 
    17 & $\frac 23$ & 1 &
    $\left(
    0,0,0,0,0,0,0,0,
    \frac 13,
    \frac 13,
    \frac 23,
    \frac 23,
    \frac 23,
    \frac 23,
    \frac 23,
    \frac 23,
    \frac 23,
    1
    \right)$
    &$-7626\frac{1}{3}$\\\hline
    18 & $\frac 23$ & 1 &
    $\left(
    0,0,0,0,0,0,0,0,0,
    \frac 13,
    \frac 13,
    \frac 23,
    \frac 23,
    \frac 23,
    \frac 23,
    \frac 23,
    \frac 23,
    \frac 23,
    1
    \right)$
    &$-14585\frac{2}{3}$\\\hline
    \end{tabular}
    }
    \end{center}
\end{table}

\newpage

Note that for $d\geq 7$ the solutions to \eqref{LP} are attained for $b_1=\ldots=b_d=1$ and hence by Proposition \ref{relaxation}, they solve the problem of minimal volumes of $V_Q(B)$. For dimensions $d\leq 6$ we needed to confirm that optimal solutions indeed extend to some quasi-copula. 
 As stated in Section \ref{sec:intro}, the realizations of $Q$ for $d\leq 4$ are already known \cite{Nelsen2002,BMUF07,UF23}, while here we need to find realizations for $d=5$ and $d=6$. Using Theorem \ref{main-v2} for this purpose, we need to find the values $q_\mi = Q(x_\mi)$ for all the vertices $x_\mi$ in the grid  $\prod_{i=1}^d \{a_1, b_1, 1\}$. We define another linear program to find these values:

\begin{align}
\label{LP-extension}
\begin{split}
\min_{
\substack{
    a_1,\ldots,a_d,\\
    b_1,\ldots,b_d,\\
    q_{\mi} \text{ for }\mi\in \{0,1,2\}^d
}}
&\hspace{1cm} \sum_{\mi\in \{0,1\}^d} \sign(\mi)\cdot q_\mi,\\
\text{subject to }
&\hspace{0.5cm} 
    0\leq a_i< b_i\leq 1\quad i=1,\ldots,d,\\
&\hspace{0.5cm} 
0\le q_{\mj} - q_{\mi} \le b_\ell - a_\ell
\quad 
    \text{for all }\ell=1,\ldots,d\\
    &\hspace{4cm}
    \text{and all }\mi\sim_\ell\mj
    \text{ with }\mj_\ell=1,\\
&\hspace{0.5cm}
0\le q_{\mj} - q_{\mi} \le 1-b_\ell
\quad 
    \text{for all }\ell=1,\ldots,d\\
    &\hspace{4cm}
    \text{and all }\mi\sim_\ell\mj
    \text{ with }\mj_\ell=2,\\
&\hspace{0.5cm} 
    \max\{0,G_d(x_{\mi})\} 
    \le q_\mi \le 
    H_d(x_{\mi})
    \quad \text{for all }\mi\in \{0,1,2\}^d.
\end{split}
\end{align}

\begin{proposition}
    \label{exact}
    Let $\mi^{(1)},\ldots,\mi^{(3^d)}$ be some order of all multi-indices $\mi\in\{0,1,2\}^d$.
    An optimal solution 
    $(a_1^\ast,\ldots,a_d^\ast,
    b_1^\ast,\ldots,b_d^\ast,
    q_{\mi^{(1)}},\ldots,q_{\mi^{(3^d)}}
    )$
    to \eqref{LP-extension}
    extends to some quasi-copula $Q:[0,1]^d\to [0,1]$
    and 
    the optimal value of \eqref{LP-extension} is the most negative volume over all boxes over all $d$-quasi-copulas.
\end{proposition}

\begin{proof}
    Use Theorem \ref{main-v2} to conclude that the optimal solution 
    to \eqref{LP-extension} indeed extends 
    to some quasi-copula.
\end{proof}

Using the linear program \eqref{LP-extension}
it turns out that the solutions from Table \ref{tab:table1}
for  $d=5$ and $d=6$ indeed extend to quasi-copulas,
since \eqref{LP-extension} has the same optimal value as \eqref{LP}.
It turned out that these extensions are not unique, but we strived for the most symmetric extension, i.e., the value of $Q$ is invariant with respect to the permutation of coordinates.
Below we state concretely both realizations, i.e., 
Examples \ref{ex:d=5} and \ref{ex:d=6} for $d=5$ and $d=6$, respectively.
Before we do so, we need some further notations.\\

Let $\calD=\prod\{a_i,b_i,1\}$ with $0<a_i<b_i<1$.
We will use multi-indices
of the form $\mi := (\mi_1, \mi_2, \ldots, \mi_d) \in \{0, 1,d+1\}^d$ to index 
$3^d$ elements of $\calD$.
The reason for this is to easily state bijective correspondence between the set of vertices of the $3$-grid and numbers $q_i$, which are attained by $Q$. Using the index set $\{0,1,2\}^d$ as before would make this correspondence more difficult to describe. 
We write 
    $\mathbf{x}_\mi:=((x_\mi)_1,\ldots,(x_\mi)_d)$ 
to denote the vertex with coordinates
\begin{equation*}
    (x_\mi)_k =
    \left\{
    \begin{array}{rl}
    a_k,& \text{if }\mi_k=0,\\
    b_k,& \text{if }\mi_k=1,\\
    1,& \text{if }\mi_k=d+1.
    \end{array}
    \right.
\end{equation*}
Note that $\|\mi\|_1$ is an integer between 0 and $d(d+1)$.
The value of $Q:\calD\to [0,1]$ in the point $x_\mi$
is 
    $q_\mi := Q(\mathbf{x}_\mi).$

\begin{ex}
\label{ex:d=5}
Let $d=5$ and 
$\calD=\prod_{i=1}^{5}\{\frac{8}{13},\frac{12}{13},1\}$.
We define
\begin{equation}
\label{val:q-5}
    q_i =
    \left\{
    \begin{array}{rl}
    \frac{1}{13},& \text{if }i=6,\\[0.5em]
    \frac{4}{13},& \text{if }i\in \{2,3,7,8\},\\[0.5em]
    \frac{5}{13},& \text{if }i\in \{12,13\},\\[0.5em]
    \frac{6}{13},& \text{if }i=18,\\[0.5em]
    \frac{8}{13},&  \text{if }i\in \{4,5,9,14,19,24\},\\[0.5em]
    \frac{9}{13},&  \text{if }i=10,\\[0.5em]
    \frac{10}{13},&  \text{if }i=15,\\[0.5em]
    \frac{11}{13},&  \text{if }i=20,\\[0.5em]
    \frac{12}{13},&  \text{if }i=25,\\[0.5em]
    1,& \text{if }i=30,\\[0.5em]
    0,& \text{for other }0\leq i\leq 30.
    \end{array}
    \right.
\end{equation}
Then $q_{\mi}:=q_{\|\mi\|_1}$ extends the function from Table \ref{tab:table1}
defined on $\prod_{i=1}^{5}\{\frac{8}{13},\frac{12}{13}\}$
to $\calD$,
and is an optimal solution to \eqref{LP-extension} with $a_1=\ldots=a_5=\frac{8}{13}$
and $b_1=\ldots=b_5=\frac{12}{13}$. By Proposition \ref{exact}, this function in turn extends to a quasi-copula.
Note that the values $i$ in the right column of \eqref{val:q-5} correspond to elements $(c_1,c_2,c_3,c_4,c_5)$ of $\calD$ in the following way:
$$
\begin{array}{|c|c|c|c|}
\hline
i & 
    \# \text{coordinates }\frac{8}{13}
  & \# \text{coordinates }\frac{12}{13}
  & \# \text{coordinates }1\\[0.2em]
\hline
2 & 3 & 2 & 0 \\\hline
3 & 2 & 3 & 0 \\\hline
4 & 1 & 4 & 0 \\\hline
5 & 0 & 5 & 0 \\\hline
6 & 4 & 0 & 1 \\\hline
7 & 3 & 1 & 1 \\\hline
8 & 2 & 2 & 1 \\\hline
9 & 1 & 3 & 1 \\\hline
10 & 0 & 4 & 1 \\\hline
12 & 3 & 0 & 2 \\\hline
13 & 2 & 1 & 2 \\\hline
14 & 1 & 2 & 2 \\\hline
15 & 0 & 3 & 2 \\\hline
18 & 2 & 0 & 3 \\\hline
19 & 2 & 1 & 3 \\\hline
20 & 0 & 2 & 3 \\\hline
24 & 1 & 0 & 4 \\\hline
25 & 0 & 1 & 4 \\\hline
30 & 0 & 0 & 5 \\\hline
\end{array}
$$
\end{ex}

\begin{remark}
Here we mention that the solution to \eqref{LP-extension} in case $d=5$ is not unique. Without implementing the condition that the
values of the solution are invariant with respect to permuting the coordinates of the multi-indices, i.e., $q_\mi=q_{\mj}$ if the coordinates of $\mj$ are obtained by permuting the ones of $\mi$, we got a different solution with the same box and the same minimal volume.
\end{remark}

\begin{ex}
\label{ex:d=6}
Let $d=6$ and 
$\calD=\prod_{i=1}^{5}\{\frac{5}{8},\frac{15}{16},1\}$.
We define 
\begin{equation}
\label{val:q-6}
    q_i =
    \left\{
    \begin{array}{rl}
    \frac{1}{16},& \text{if }i=8,\\[0.5em]
    \frac{1}{8},& \text{if }i=14,\\[0.5em]
    \frac{5}{16},&   \text{if }i\in \{3,4,9,15,21\},\\[0.5em]
    \frac{3}{8},&    \text{if }i=10,\\[0.5em]
    \frac{7}{16},&    \text{if }i=16,\\[0.5em]
    \frac{1}{2},&    \text{if }i\in\{22,28\},\\[0.5em]
    \frac{5}{8},&    \text{if }i\in\{5,11,17,23,29\},\\[0.5em]
    \frac{11}{16},&    \text{if }i=12,\\[0.5em]
    \frac{3}{4},&    \text{if }i=18,\\[0.5em]
    \frac{13}{16},&    \text{if }i=24,\\[0.5em]
    \frac{7}{8},&    \text{if }i=30,\\[0.5em]
    \frac{15}{16},&    \text{if }i=36,\\[0.5em]
    1,&    \text{if }i=42,\\[0.2em]
    0,& \text{otherwise}.
    \end{array}
    \right.
\end{equation}
Then $q_{\mi}:=q_{|\mi|}$ extends the function from Table \ref{tab:table1}
defined on $\prod_{i=1}^{6}\{\frac{5}{8},\frac{15}{16}\}$
to $\calD$,
and is an optimal solution to \eqref{LP-extension} with $a_1=\ldots=a_6=\frac{5}{8}$
and $b_1=\ldots=b_6=\frac{15}{16}$. By Proposition \ref{exact}, this function in turn extends to a quasi-copula.
Note that the values $i$ in the right column of \eqref{val:q-6} correspond to elements $(c_1,c_2,c_3,c_4,c_5,c_6)$ of $\calD$ in the following way:
$$
\begin{array}{|c|c|c|c|}
\hline
i & 
    \# \text{coordinates }\frac{5}{8}
  & \# \text{coordinates }\frac{15}{16}
  & \# \text{coordinates }1\\[0.2em]
\hline
3 & 3 & 3 & 0 \\\hline
4 & 2 & 4 & 0 \\\hline
5 & 1 & 5 & 0 \\\hline
7 & 5 & 0 & 1 \\\hline
8 & 5 & 1 & 1 \\\hline
9 & 3 & 2 & 1 \\\hline
10 & 2 & 3 & 1 \\\hline
11 & 1 & 4 & 1 \\\hline
12 & 0 & 5 & 1 \\\hline
14 & 5 & 0 & 2 \\\hline
15 & 3 & 1 & 2 \\\hline
16 & 2 & 2 & 2 \\\hline
17 & 1 & 3 & 2 \\\hline
18 & 0 & 4 & 2 \\\hline
21 & 3 & 0 & 3 \\\hline
22 & 1 & 3 & 2 \\\hline
23 & 1 & 2 & 3 \\\hline
24 & 0 & 3 & 3 \\\hline
28 & 2 & 0 & 4 \\\hline
29 & 1 & 1 & 4 \\\hline
30 & 0 & 2 & 4 \\\hline
36 & 0 & 1 & 5 \\\hline
42 & 0 & 0 & 6 \\\hline
\end{array}
$$
\end{ex}

\newpage

\subsection{Maximal volume problem}
\label{subsec:maximal}
Let us define the following linear program:

\begin{align}
\label{LP-max}
\begin{split}
\max_{
\substack{
    a_1,\ldots,a_d,\\
    b_1,\ldots,b_d,\\
    q_{\mi} \text{ for }\mi\in \{0,1\}^d
}}
&\hspace{0.2cm} \sum_{\mi\in [1]^d} \sign(\mi)\cdot q_\mi,\\
\text{subject to }
&\hspace{0.2cm} 
    0\leq a_i< b_i\leq 1\quad i=1,\ldots,d,\\
&\hspace{0.2cm} 
0\le q_{\mj} - q_{\mi} \le b_\ell - a_\ell
\quad 
    \text{for all }\ell=1,\ldots,d
    \text{ and all }\mi\sim_\ell\mj,\\
&\hspace{0.2cm} 
    \max\{0,G_d(x_{\mi})\} 
    \le q_\mi \le 
    H_d(x_{\mi})
    \quad \text{for all }\mi\in \{0,1\}^d.
\end{split}
\end{align}

\begin{proposition}
    \label{relaxation-max}
    Let $\mi^{(1)},\ldots,\mi^{(2^d)}$ be some order of all multi-indices $\mi\in\{0,1\}^d$.
    If the optimal solution 
    \begin{equation}
    \label{optimal-soln-max}
    (a_1^\ast,\ldots,a_d^\ast,
    b_1^\ast,\ldots,b_d^\ast,
    q_{\mi^{(1)}},\ldots,q_{\mi^{(2^d)}}
    )
    \end{equation}
    to \eqref{LP-max}
    satisfies 
    \begin{equation}
        \label{bi-equal-to-1-max}
            b_1^\ast=\ldots=b_d^\ast=1,
    \end{equation}
    then the optimal value of \eqref{LP-max} is the most positive volume over all boxes over all $d$-quasi-copulas.
\end{proposition}

\begin{proof}
    Use Theorem \ref{main} for the optimal solution 
    \eqref{optimal-soln-max} to see that this solution indeed extends 
    to some quasi-copula.
\end{proof}

By Proposition \ref{relaxation-max}, the linear program
\eqref{LP-max}
relaxes the problem of determining a quasi-copula $Q$,
such that the volume $V_Q(\cB)$ is maximal among all
$d$-quasi-copulas and all boxes $\cB\subseteq [0,1]^d$.
The relaxation refers to the fact, that if \eqref{bi-equal-to-1-max} does not hold, then the optimal solution to
\eqref{LP-max} might not extend to some quasi-copula and hence the optimal value of \eqref{LP-max} represents the upper bound for the maximal volume problem.\\

As in the case of the minimal volume problem we used computer software 
\textit{Julia} to compute the optimal solutions to \eqref{LP-max}  up to $d= 17$. The execution time was $213\text{s}$ for $d=15$, $75\text{s}$ for $d=16$ and $3073\text{s}$ for $d=17$ using a personal computer with \emph{AMD Ryzen 7 6800HS} processor and $16\mathrm{GB}$ of RAM. 
Note that the execution time decreased when increasing dimension from $d=15$ to $d=16$. This is due to the fact that the methods used are numerical and as such also depend on some random steps in the procedure. In general larger dimension means larger problem and longer execution time but such anomalies as in this case might happen. 
In Table \ref{tab:table2} below we list the results obtained (see  Gitlab repository \cite{Quasi-copula-Gitlab}).
\newpage

\begin{table}
  \begin{center}
      \caption{Maximal values of $V_Q(B)$ over all $d$--variate quasi--copulas $Q$ and all $d$--boxes $B\subseteq [0,1]^d$. 
      It turns out that the maximal box $B_{\max}$ is of the form $[a,b]^d$ and that the values $q_{\mi}$, $\mi=(\mi_1,\ldots,\mi_d)\in \{0,1\}^d$,
      depend only on 
      $\|\mi\|_1=\sum_{j=1}^d \mi_j$.
      We list $q_{\|\mi\|_1}:=q_\mi$.
      }
    \label{tab:table2}
    {\renewcommand{\arraystretch}{2}
    \begin{tabular}{|c|c|c|c|c|}
    \hline
    $d$ & $a$ & $b$ & 
    $\vec{q}=(q_{\|\mi\|_1})_{\|\mi\|_1=0}^d$ & $V_Q([a,b]^d)$ 
    \\
    \hline
    2 & 
        0 & 1 & 
        $\left(0,0,1\right)$
        &
        1  \\
    \hline
    3 & 
        0 & 1 &
        $\left(0,0,0,1\right)$
        & 1 \\
    \hline
    4 & 
        $\frac 12$ & $1$ & 
        $\left(0,0,\frac 12,\frac 12,1\right)$
        &$2$\\
    \hline
    5 & 
        $\frac 12$ & $1$ & 
        $\left(0,0,0,\frac 12,\frac 12,1\right)$
        &
        $3\frac{1}{2}$  \\
    \hline
    6 & 
        $\frac 12$ & $1$ 
        &
        $\left(0,0,0,0,\frac 12,\frac 12,1\right)$
        & $5\frac{1}{2}$ \\
    \hline
    7 & 
        $\frac 23$ & $1$ & 
        $\left(0,0,0,\frac 13,\frac 13,\frac 23,\frac 23,1\right)$
        &   $10\frac{1}{3}$  \\
    \hline
    8 & 
        $\frac 23$ & $1$ & 
        $\left(0,0,0,0,
        \frac 13,\frac 13,\frac{2}{3},\frac 23,1\right)$
        &
        $19$ \\
    \hline
    9 & 
        $\frac 12$ & $1$ & 
        $\left(0,0,0,0,0,
        \frac 12,\frac 12,\frac 12,\frac 12,1\right)$
        & $35\frac 12$ \\
    \hline
    10 & 
        $\frac 23$ & 1 & 
        $\left(0,0,0,0,\frac 13,\frac 13,\frac{2}{3},\frac 23,\frac{2}{3},\frac 23,1\right)$
        &
        $ 70\frac{1}{3}$ \\
    \hline
    11 & 
        $\frac 23$ & 1 & 
        $\left(0,0,0,0,0,
        \frac 13,\frac 13,
        \frac 23,\frac 23,
        \frac 23,\frac 23,
        1\right)$
        &
        $140\frac{1}{3}$ \\
    \hline
    12 & 
        $\frac 23$ & 1 & 
        $\left(0,0,0,0,0,0,
        \frac 13,\frac 13,
        \frac 23,
        \frac 23,
        \frac 23,
        \frac 23,
        1\right)$
        &
        $264\frac{1}{3}$ \\
    \hline
    13 & 
        $\frac 34$ & 1  
        &
        $\left(
        0,0,0,0,0,
        \frac 14, 
        \frac 14, 
        \frac 12, 
        \frac 12, 
        \frac 34, 
        \frac 34, 
        \frac 34, 
        \frac 34, 
        1
        \right)$
        &
        $478\frac{3}{4}$ \\
    \hline
    14 & 
        $\frac 23$ & 1 
        &
        $\left(
        0,0,0,0,0,0,
        \frac 13, 
        \frac 13, 
        \frac 23, 
        \frac 23, 
        \frac 23, 
        \frac 23, 
        \frac 23, 
        \frac 23, 
        1
        \right)$
        & $1001\frac{1}{3}$ \\
    \hline
    15 & 
        $\frac 23$ & 1 & 
        $\left(
        0,0,0,0,0,0,0,
        \frac 13, 
        \frac 13, 
        \frac 23, 
        \frac 23, 
        \frac 23, 
        \frac 23, 
        \frac 23, 
        \frac 23, 
        1
        \right)$
        &$2002\frac{1}{3}$ \\
    \hline
    16 & 
        $\frac 23$ & 1 & 
        $\left(
        0,0,0,0,0,0,0,0,
        \frac 13, 
        \frac 13, 
        \frac 23, 
        \frac 23, 
        \frac 23, 
        \frac 23, 
        \frac 23, 
        \frac 23, 
        1
        \right)$
        & $3813\frac 23$ \\\hline
    17 & $\frac 23$ & 1 &
    $\left(
    0,0,0,0,0,0,0,
    \frac 14,
    \frac 14,
    \frac 12,
    \frac 12,
    \frac 34,
    \frac 34,
    \frac 34,
    \frac 34,
    \frac 34,
    \frac 34,
    1
    \right)$
    &$7221\frac{3}{4}$\\
    \hline
    \end{tabular}
    }
    \end{center}
\end{table}

Note that for $d\geq 2$ the solutions to \eqref{LP-max} are attained for $b_1=\ldots=b_d=1$ and hence by Proposition \ref{relaxation-max}, they indeed solve the problem of maximal volumes of $V_Q(B)$. There is no need for additional consideration of low-dimensional cases as compared to the minimal volume problem and $d\leq 6$.


\subsection{Graphical presentation of solutions and concluding remarks}
\label{subsec:concluding}


Results in Tables \ref{tab:table1} and \ref{tab:table2} above solve \cite[Open Problem 5]{ARIASGARCIA20201} up to $d=18$ for the minimal values and $d=17$ for the maximal values. Our results agree with \cite{Nelsen2002} for $d=2$, with \cite{BMUF07} for $d=3$ and with \cite{UF23} for $d=4$.
We disprove the conjecture from \cite{UF23}, which states that
the minimal value of $V_Q(\cB)$ in dimension $d$ is $-\frac{(d-1)^2}{2d-1}$, attained for some $Q$ on $\cB=\left[\frac{d-1}{2d-1},\frac{2d-2}{2d-1}\right]^d$.
Our results indicate that the growth is exponential of the form $c2^{d}+d$ for suitable $c,d\in \RR$
(see Figure \ref{fig:figure1} below), while boxes are of the form $[a,1]^d$, $a\in (0,1)$, for $d\geq 7$.

\begin{center}
\begin{figure}[h!]
      \caption{
      Blue (and orange) points represent the volumes of boxes with maximal (and minimal) volume over all $d$-quasi-copulas and all $d$-boxes. The $x$-axis represents the dimension $d$, while the $y$-axis represents the value of the volume in thee logarithmic scale with base 2.}
\includegraphics[width=12cm]{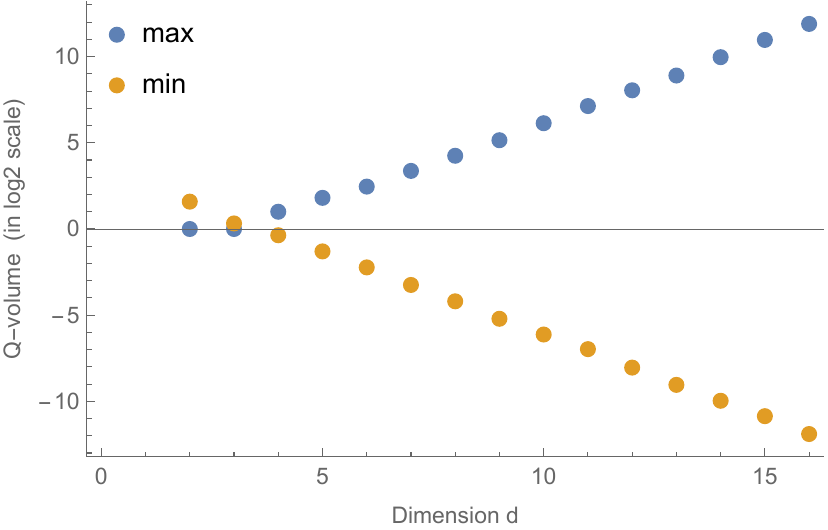}
\label{fig:figure1}
\end{figure}
\end{center}


The approach via linear programming to obtain results for $d$ larger than 18 does not seem to be feasible due to computer memory limits. However, we believe that a lot information about the true behavior of the extremes of the $Q$-volume is already captured in dimensions $d\leq 18$. In future, our perspective is to determine the formulas for extreme values in terms of the dimension $d$.  

The presented results suggest that there could be a general closed form formula for the minimal and maximal $Q$-volumes. 
Splitting the results according to the remainder of the division of $d$ with 4, it seems that
the formulas for the minimal volume for $d\geq 7$ are 
\begin{equation}
\label{conj:minimal}
\left\{
\begin{array}{rl}
    \frac{1}{3}\left[1- \binom{d}{\frac{d}{2}-1}\right],&   \text{if  }d\text{ mod } 4\in \{0,2\},\\[0.5em]
    \frac{1}{3}\left[1-2\binom{d-1}{\frac{d-1}{2}+1}\right],&   \text{if  }d\text{ mod } 4=1,\\  [0.5em]  
    \frac{1}{4}\left[1-3\binom{d-1}{\frac{d+3}{2}}+\binom{d}{\frac{d+3}{2}}
        -
        \binom{d}{\frac{d+1}{2}}\right],&   \text{if  }d\text{ mod } 4=3,\\  
\end{array}    
\right.
\end{equation}
and for the maximal volume for $d\geq 2$ they are
\begin{equation}
\label{conj:maximal}
\left\{
\begin{array}{rl}
    \frac{1}{3}\left[1+ \binom{d}{\frac{d}{2}-1}\right],&   \text{if  }d\text{ mod } 4\in \{0,2\},\\[0.5em]
           \frac{1}{4}\left[1+3\binom{d-1}{\frac{d+1}{2}}
        -\binom{d}{\frac{d+3}{2}}
        +
        \binom{d}{\frac{d+1}{2}}\right],&   \text{if  }d\text{ mod } 4=1,\\  [0.5em]
    \frac{1}{3}\left[1+2\binom{d-1}{\frac{d-1}{2}+1}\right],&   \text{if  }d\text{ mod } 4=3.
\end{array}    
\right.
\end{equation}
Let us explain how we came to this formulas. Observing the pattern of solutions in Tables \ref{tab:table1} and \ref{tab:table2} we can construct quasi-copulas and the boxes with these volumes. 
Let us demonstrate how to obtain a formula for $d\text{ mod } 4=0$ in \eqref{conj:minimal}.

\begin{ex}
\label{ex:min-4}
Let $d=4k$ with $k\in \NN$, $k\geq 2$ and $\cB=[\frac{2}{3},1]^d$. For $\mi \in \{0,1\}^d$ we define
$$
q_{\mi}
=
\left\{
    \begin{array}{rl}
        1,&             \text{if } \|\mi\|_1=d,\\[0.5em]
        \frac{2}{3},&   \text{if } \frac{d}{2}+1\leq \|\mi\|_1 \leq d-1,\\[0.5em]
        \frac{1}{3},&   \text{if } \|\mi\|_1\in \left\{\frac{d}{2},\frac{d}{2}-1\right\},\\[0.5em]
        0,&             \text{otherwise}.
    \end{array}
\right.
$$
Then for the quasi-copula obtained by Theorem \ref{main} we have that
    $$V_Q\Big(\Big[\frac{2}{3},1\Big]^d\Big)=
    \frac{1}{3}\left[1- \binom{d}{\frac{d}{2}-1}\right].$$
This is true by the following computation:
\begin{align*}
&V_Q\Big(\Big[\frac{2}{3},1\Big]^d\Big)=\sum_{
    \mi\in \{0,1\}^d
                } \sign(\mi)q_\mi\\
&=
q_{(1,\ldots,1)}
+
\sum_{
    \substack{\mi\in \{0,1\}^d,\\
                \frac{d}{2}+1\leq \|\mi\|_1\leq d-1
                }
                } (-1)^{d-\|\mi\|_1}q_\mi 
+
\sum_{
    \substack{\mi\in \{0,1\}^d,\\
                \|\mi\|_1\in \{\frac{d}{2},\frac{d}{2}-1\}
                }
                } (-1)^{d-\|\mi\|_1}q_\mi\\
&=
1-
    \frac{2}{3}
    \left[        
        \binom{d}{d-1}
        -
        \binom{d}{d-2}
        +\ldots
        +\binom{d}{\frac{d}{2}+1}
    \right]
+\frac{1}{3}
    \left[
        \binom{d}{\frac{d}{2}}-\binom{d}{\frac{d}{2}-1}    
    \right]\\
&=^{(1)}
\frac{1}{3}
+\frac{1}{3}
    \left[
        \binom{d}{d}
        -
        \binom{d}{d-1}
        +\ldots
        -\binom{d}{1}+\binom{d}{0}
    \right]
-\frac{1}{3}\binom{d}{\frac{d}{2}-1}\\
&=^{(2)}\frac{1}{3}\left[1- \binom{d}{\frac{d}{2}-1}\right],
\end{align*}
where in (1) we used 
    $1=\frac{1}{3}+\frac{1}{3}\binom{d}{d}+\frac{1}{3}\binom{d}{0}$
and put the second and the third summand in the middle bracket,
decomposed each summand from the middle bracket into 
    $\binom{d}{i}=\frac{1}{2}\binom{d}{i}+\frac{1}{2}\binom{d}{d-i}$,
and put $\frac{1}{3}\binom{d}{\frac{d}{2}}$ from the last bracket to the middle bracket, while in (2) we used that $0=(1-1)^d=\sum_{i=0}^d \binom{d}{i}(-1)^i.$\hfill$\blacksquare$
\end{ex}

Similarly as in Example \ref{ex:min-4} there are concrete realizations for other cases in \eqref{conj:minimal}
and \eqref{conj:maximal}. This shows that the minimal volume values are at most the ones stated in \eqref{conj:minimal}, while the maximal volume values are at least the ones stated in \eqref{conj:maximal}.   However, the linear programming approach we used to obtain the results does not offer further insights how to prove correctness of these formulas. 
It seems the solutions in dimensions $d\text{ mod } 4=3$ follow a different pattern than other solutions in case of minimal volumes, while for maximal volumes this observation applies to 
dimensions $d\text{ mod } 4=1$. It would be also interesting to explain the special behaviour of the minimal volume in dimensions $d<7$ and figure out whether there is some functional relation between the minimal and the maximal volume in each dimension.

To settle the  observations of the previous paragraph a software based linear programming approach 
is not appropriate, but a more analytical investigation is required. 
At this point such investigation seems to be difficult to carry out and will require additional insights
into the problem.

To conclude let us also mention there are other interesting special classes $\widetilde \cC_d$ of $d$-quasi-copulas,
strictly sandwiched betweeen the classes of $d$-copulas $\cC_d$ and $d$-quasi-copulas $\cQ_d$, i.e., 
    $\cC_d\subset \widetilde \cC_d \subset \cQ_d$. Examples of such classes are \emph{supermodular} quasi-copulas 
    and \emph{$k$-dimensional increasing} quasi-copulas for $k\in \{2,\ldots,d\}$ 
    (see \cite{ARIASGARCIA20201,AGMB17}).  It would be interesting to study extremal volume question for these classes as well in the future research.

\bigskip

 \noindent \textbf{Acknowledgement.}\
We would like to thank three anonymous reviewers for carefully reading our manuscript and
many suggestions to improve the overall presentation of our results.

 \end{document}